\documentclass[11pt,reqno,section]{amsart}
\usepackage[english]{babel}
\usepackage[utf8]{inputenc}
\usepackage{amsmath}
\usepackage{amsthm}
\usepackage{verbatim}
\usepackage{enumerate}
\usepackage{amssymb}
\usepackage{amsfonts}
\usepackage{amscd,bezier}
\usepackage{anysize}
\usepackage{indentfirst}
\usepackage{upref}
\usepackage{color}
\usepackage[pagewise]{lineno}
\usepackage{tikz}
\usepackage{theoremref}
\usetikzlibrary{matrix}
\usepackage{pstricks}
\usepackage{subfig}
\usepackage{graphicx}
\usepackage{pstricks}
\usepackage{amscd}
\usepackage{relsize}
\usepackage{empheq}
\usepackage[toc,page]{appendix}
\usepackage{commath}
\usepackage{mathabx}
\usepackage{MnSymbol}
\usepackage{xcolor}
\usepackage[hidelinks]{hyperref}
\allowdisplaybreaks
\numberwithin{equation}{section}

\newtheorem{theorem}{Theorem}[section]
\newtheorem{lema}{Lemma}[section]
\newtheorem{remark}{Remark}[section]
\newtheorem{de}{Definition}[section]
\newtheorem{pr}{Proposition}[section]
\newtheorem{co}{Corollary}[section]
\newtheorem{example}{{Example}}[section]

\marginsize{3cm}{3cm}{3cm}{3cm}

\newcommand{\real}{\mathbb{R}}
\newcommand{\nat}{\mathbb{N}}
\newcommand{\bb}{\mathcal{B}}

\newcommand{\lc}{{\mathcal L}}

\newcommand{\lx}{{\mathcal L}(X)}

\newcommand{\xr}{x_{r(t)}}

\newcommand{\xt}{x_{t}}
\newcommand{\vp}{\varphi}
\newcommand{\ta}{\theta}
\newcommand{\ve}{\varepsilon}
\newcommand{\al}{\alpha}
\newcommand{\sa}{\sigma}
\newcommand{\La}{\Lambda}
\newcommand{\la}{\lambda}
\newcommand{\om}{\omega}
\newcommand{\Om}{\Omega}

\title[Abstract Functional differential equations]
{Evolution families and variation of constants formula for abstract functional differential equations with time-dependent infinite delay.}

\author[C. Carrasco]{Claudio Carrasco $^{a}$}
\address{$^{a}$Universidad de Chile, Departamento de Matem\'aticas.
	Casilla 653, Santiago, Chile.}
\email{claudio.carrasco.g@ug.uchile.cl} \thanks{}

\author[C. A. Gallegos]{Claudio A. Gallegos$^{b}$}
\address{$^{b}$ Universidad de Chile, Departamento de Matem\'aticas.
	Casilla 653, Santiago, Chile.}
\email{claudiogallegos@uchile.cl} \thanks{}
    
\author[H. R. Henr\'{\i}quez]{Hern\'an R. Henr\'{\i}quez$^{c}$}
\address{$^{c}$Universidad de Santiago, USACH, Departamento de Matem\'atica y Ciencia de la Computaci\'on, Casilla 307,
	correo 2, Santiago, Chile. }
\email{hernan.henriquez@usach.cl}
\author[M. F. Pinaud]{Matthieu F. Pinaud$^{d}$}\address{$^{d}$ Universidad de Santiago, USACH, Departamento de Matem\'atica y Ciencia de la Computaci\'on, Casilla 307,
	correo 2, Santiago, Chile.}
\email{matthieu.pinaud@usach.cl} \thanks{
	}

\begin{document}

\begin{abstract}
In this paper, we consider a class of first-order abstract retarded functional differential equations in Banach spaces, incorporating a time-dependent infinite delay governed by a regulated function. We establish the existence of mild solutions for the nonlinear equation and show that the family of solution maps for the linear equation forms a well-defined evolution family of bounded linear operators on an appropriate phase space. Furthermore, we leverage this evolution family to prove a variation of constants formula for the inhomogeneous linear problem.
\end{abstract}

\dedicatory{Dedicated to the memory of Hern\'an R. Henr\'iquez.}



\maketitle

\setcounter{equation}{0}
\setcounter{pr}{0}
\setcounter{de}{0}
\setcounter{theorem}{0}
\setcounter{lema}{0}
\setcounter{co}{0}
\setcounter{example}{0}

\pagestyle{myheadings} \markboth{\hfil  C. Carrasco, C. A. Gallegos,  H. R. Henr\'{i}quez and M. F. Pinaud
\hfil $\hspace{3cm}$ } {\hfil$\hspace{1.5cm}$
{Abstract  functional differential equations }
\hfil}

\section{Introduction.} \label{introd}
Abstract retarded functional differential equations (ARFDEs) provide the natural mathe\-matical framework for hereditary systems, where the state evolution depends on the past history. These equations naturally arise in the modeling of viscoelastic materials, structured population dynamics, and thermal conduction with fading memory, among other physical phenomena; for the general theory of these framework and its diverse applications, we refer the reader to the classical monographs \cite{BaP,BDDM,GO,HA1,KM,McK,vlasov,Wu}.

In this work, we study the existence of mild solutions for ARFDEs featuring time-dependent infinite delay. Motivated by the recent framework established in \cite{HMH}, we focus our attention on the semilinear abstract Cauchy problem
\begin{empheq}[left=\empheqlbrace]{align}
\quad x'(t) &= Ax(t) + f(t, x_{r(t)}), \quad 0 \leq t \leq a, \label{equ1.2} \\
 \quad x_0 &= \varphi, \label{equ1.3}
\end{empheq}
where $a > 0$, $x(t)$ takes values in a Banach space $X$ endowed with the norm $\|\cdot\|$, $r\colon [0, a] \to \mathbb{R}$ is a regulated function satisfying $r(t) \leq t$ for all $t \in [0,a]$, and $A\colon D(A) \subseteq X \to X$ is the infinitesimal generator of a strongly continuous semigroup of bounded linear operators $(T(t))_{t \geq 0}$ on $X$. Moreover, as usual, for any function $x\colon(-\infty, a] \to X$ and each $t \in [0, a]$, the history function $x_t\colon (-\infty, 0] \to X$ is defined by $x_t(\theta) = x(t + \theta)$, for all $\theta \leq 0$. To analyze this problem in the infinite delay context, it is crucial to specify an appropriate phase space $\mathcal{B}$ for initial conditions $\varphi$, which contains these history functions $x_t$.

It is worth noting that $f(t,\xr)$ can be rewritten in a simple non-autonomous form by defining a function $g(t,\cdot)$ such that
\begin{equation}\label{eq: transformation}
    f(t,\xr) = g(t,\xt), \quad t \geq 0.
\end{equation}

Indeed, this is achieved by setting $g(t,\vp) = f(t,\vp_{r(t)-t})$. However, this transformation introduces two major drawbacks for the development of the theory. First, it requires ensuring that $\vp_{r(t)-t}$ belongs to the phase space $\bb$. Second, the time dependence of $g(t,\vp)$ becomes as intricate as that of $\vp_{r(t)-t}$ itself. Moreover, as will be discussed in Section~\ref{prel}, the function $\vp_{r(t)-t}$ does not, in general, belongs to the phase space $\bb$.

\subsection{Main contributions and outline of the work}

As a first objective, this work establishes new criteria for the existence of mild solutions to the semilinear problem \eqref{equ1.2}--\eqref{equ1.3}. Specifically, we complement and extend the recent framework of \cite{HMH}, including existence results under general conditions as well as under the assumption that the associated semigroup is compact or immediately norm-continuous, see Theorems~\ref{T7}, \ref{T3.1} and Corollaries~\ref{C3.1}, \ref{C3.2}.

As a concrete application, we consider a partial integro-differential equation describing a diffusion process with memory and a time-dependent delay in the one-dimensional spatial domain $[0, \pi]$. This system is described by the initial-boundary value problem
\begin{empheq}[left=\empheqlbrace]{align}
\frac{\partial w(t, \xi)}{\partial t} & =  \frac{\partial^{2} w(t, \xi)}{\partial \xi^{2}} + \int_{-\infty}^{t} P(s-t) w(s- t + r(t), \xi) \, d s + h_{0}(t, \xi), \quad 0 \leq t \leq a,  \label{equ1.4} \\
w(t, 0) &= w(t, \pi) = 0, \quad 0 \leq t \leq a, \label{equ1.5} \\
w(\theta, \xi) & =  \varphi(\theta, \xi), \quad - \infty < \theta \leq 0, \label{equ1.6}
\end{empheq}
where $w\colon (-\infty, a] \times [0, \pi] \to \mathbb{R}$ represents the temperature distribution. Here, $P\colon(-\infty, 0] \to \mathbb{R}$ is a Lebesgue measurable function, and $h_{0}$, $r$, and $\varphi$ are suitable functions whose properties are detailed in Section~\ref{applications}.

The central contribution of this paper lies in the study of the abstract linear structure associated with the model described above. Specifically, we analyze the linear non-homogeneous ARFDE with time-dependent infinite delay of type
\begin{empheq}[left=\empheqlbrace]{align}
x'(t) &= A x(t) + \Lambda(x_{r(t)}) + h(t), \quad 0 \leq s \leq t \leq a, \label{equ4.1}\\
x_{s} &= \varphi, \label{equ4.2}
\end{empheq}
where $\Lambda : \mathcal{B} \to X$ is a bounded linear operator and $h : [0, a] \to X$ is integrable. By restricting to an appropriate phase space $\mathcal{B}_{\tau}$ containing the initial condition $\varphi$, we construct a well-defined evolution family $U_{\Lambda}(t,s)$ on $\mathcal{B}_{\tau}$ through a Dyson--Phillips type series, see Proposition~\ref{serie D-P}. Moreover, we prove that this evolution family yields the mild solution of the homogeneous problem, establishing the fundamental identification of the state history $x_t$ with the action $U_{\Lambda}(t,s)\varphi$, see Theorem~\ref{P4.1}. In addition, our result characterizes the underlying structural properties of the family, such as its translation property and its formulation as a perturbed integral equation. These developments serve as the essential foundation to derive our main result Theorem~\ref{equ4.6}: an explicit variation of constants formula for the non-homogeneous system. 

To the best of our knowledge, this is the first such representation formula reported in the literature for this class of equations with time-dependent infinite delay.

\medskip

\subsection{Terminology  and notations}

Throughout this manuscript, we adopt the standard conventions and framework of functional analysis. In particular, if $(Z,\|\cdot \|_{Z})$ and
$(Y, \|\cdot \|_{Y})$ are Banach spaces, we denote by
${\mathcal{L}}(Z,Y)$ the Banach space of bounded linear operators
from $Z$ into $Y$ endowed with the norm of operators, and
we abbreviate this notation to ${\mathcal{L}}(Z)$ whenever $Z=Y$.
Throughout  this paper, $I = [0, a]$ for $a > 0$ is a fixed interval, and
$C(I,Z) $ is the space of continuous functions
from $I$ into $Z$ endowed with the norm of the  uniform
convergence. Similarly, $C_{b}((-\infty, 0], Z) $ denotes the space of bounded continuous functions from
$(-\infty, 0]$ into $Z$ endowed with the norm of uniform convergence, and $C_{0}((-\infty, 0], Z) $ is the subspace of
$C_{b}((-\infty, 0], Z) $ consisting of functions $x$ such that $x(\ta) \to 0$ as $\ta \to - \infty$.
 Moreover, $B_{R}(x,Z)$ denotes the
closed ball with center $x$ and radius $R$ in $Z$. When  the space $Z$ is clearly
determined from the context, we  abbreviate this notation
to  $B_{R}(x)$.

\setcounter{equation}{0}
\setcounter{remark}{0}
\setcounter{pr}{0}
\setcounter{de}{0}
\setcounter{theorem}{0}
\setcounter{lema}{0}
\setcounter{co}{0}
\setcounter{example}{0}

\setcounter{equation}{0}
\setcounter{remark}{0}
\setcounter{pr}{0}
\setcounter{de}{0}
\setcounter{theorem}{0}
\setcounter{lema}{0}
\setcounter{co}{0}
\setcounter{example}{0}

\section{Preliminaries} \label{prel}

In this section we introduce the concept of phase space, and we review some fundamental
properties  about the abstract Cauchy problem of first order which are needed to establish
our results.

To study retarded functional differential equations with infinite delay we need to introduce
the concept of phase space.
In this paper, we use the axiomatic definition of phase space
$\mathcal{B}$ introduced by Hino et al. in \cite{Hi}.
 Thus, $\bb$ will be a linear space
of functions mapping $ \; ( - \infty , 0]$ into $ X $ endowed with
a norm $\| \cdot \|_{\bb}$.
 We will assume that $\bb$ satisfies the following axioms: 
\begin{description}
\item [(A)] If $  x: (- \infty , \sigma + b) \rightarrow X$, $\sigma \in \real$, $b > 0,
$ is continuous on $  [ \sigma , \sigma + b)$ and $x_{\sigma } \in
\bb$, then for every $ t  \in [ \sigma , \sigma + b) $ the
following conditions hold:
\item [(i)] $ x_{t} \in \bb$.
\item [(ii)] $ \| x(t) \| \leq H \| x_{t} \|_{\bb}$, where $H \geq 0$ is a constant independent of $x$.
\item [(iii)] $ \|x_{t}\|_{\bb} \leq K (t - \sigma )\, \sup \{ \|x(s)\| :
\sigma \leq s \leq t \} \;+\; M (t - \sigma )
\|x_{\sigma}\|_{\bb}$, where $K, M : [0,
\infty ) \rightarrow [0, \infty )$, $ K $ is continuous and $M$ is
locally bounded. The functions $K $ and $  M $ are independent of $x$.
\item [(A-1)] For the function $x $ in \textbf{(A)}, the map $t \mapsto x_{t}  $
is a  $\bb$-valued continuous function on $[\sigma, \sigma + b)$.
\item [(B)] The space $\bb$ is complete.
\end{description}
Throughout this paper we always  assume that $\bb$ is  a phase
space.

\begin{remark}\rm\label{R2.1}
	In the theory of retarded functional differential equations with
	unbounded delay frequently we need additional properties of the
	space  $\bb$ to obtain  some results. Next we denote by $C_{00}$
	the space of continuous functions  from $(- \infty, 0]$ into $X$ with
	compact support. It is clear from the axiom \textbf{(A)-(i)} of phase space that
	$C_{00} \subseteq \bb$. 
\end{remark}

In this work we consider the following axiom (\cite{Hi}):

\begin{description}
\item[(C-2)] If a uniformly bounded sequence $(\vp^{n})_{n}$ in $C_{00}$
converges to a function $\vp$ in the compact-open topology, then $\vp$
belongs to $\bb$ and $\|\vp^{n} - \vp\|_{\bb} \to 0$, as $n
\to \infty$.
\end{description}

\begin{remark} \label{R2.3}
When the axiom \textup{\textbf{(C-2)}} holds, the space $C_{b}((-\infty, 0], X) $ is
continuously included in $\bb $ {\rm (\cite[Proposition~7.1.1]{Hi})}.
Thus, there is a constant   $Q \geq 0$ such that $\| \psi \| _{\bb} \leq Q
\|\psi\|_{\infty }$, for all $\psi \in
C_{b}((-\infty ,0],X)$.
\end{remark}

\begin{example} \label{Ex1} The phase space $C_{g}^{0}(X)$.
\rm Let $\bb = C_{g}^{0}(X)$ be the space consisting of
continuous functions $\vp: (-\infty, 0]
 \to X$ such that $ \;{\displaystyle \lim_{\theta \to -\infty}
 \frac{\|\varphi(\theta)\|}{g(\theta)}
= 0 }$, where  $g : (- \infty, 0] \to (0, \infty)$ is  a
continuous function that satisfies conditions (g-1) and (g-2) in
the terminology of \cite{Hi}. This means that
\begin{itemize}
\item[(g-1)] The function $\;{\displaystyle G(t) = \sup_{\theta \leq -t} \frac{g(t +
\theta)}{g(\theta)}} $ is locally bounded for $t \geq 0$.
\item[(g-2)] $g(\theta) \to \infty$ as $\theta \to - \infty$.
\end{itemize}
To facilitate some estimates, we always assume that $g$ is decreasing.
With the norm in $\bb$  defined by
\[
\|\varphi \|_{\bb} = \sup_{\theta \leq 0}
\frac{\|\varphi(\theta)\|}{g(\theta)}
\]
the space $\bb$ is a phase space (\cite[Theorem~1.3.2]{Hi}).
\end{example}

\begin{example} \label{Ex3} \rm In this example we consider a particular case of Example~\ref{Ex1}.
We take $X = \real$ and $g(\ta) = e^{\gamma \ta^{2}}$, $ \ta \leq 0$, for a constant $\gamma > 0$. It is easy to see
that conditions (g-1) and (g-2) are satisfied. Let $\vp : (-\infty, 0] \to \real$ be given by
\[
\vp(\ta) = \left\{ \begin{array}{ccl} e^{-\gamma}, & & -1 \leq \ta \leq 0, \\
e^{\gamma( \ta^{2} + 2 \ta)}, & &  \ta \leq -1.
\end{array} \right.
\]
Since
\[
\frac{\vp(\ta)}{g(\ta)} =   e^{2 \gamma  \ta} \to 0, \; \ta \to - \infty,
\]
we infer that $\vp \in C_{g}^{0}(\real)$. Moreover, $\vp_{-1} \not\in C_{g}^{0}(\real)$. In fact,
\[
\frac{\vp_{-1}(\ta)}{g(\ta)} = \frac{\vp(\ta -1)}{g(\ta)} =  e^{-  \gamma},
\]
for $\ta$ large enough.

This shows that the transformation in \eqref{eq: transformation} can not be applied in the phase space
$C_{g}^{0}(\real)$.
\end{example}

\begin{example} \label{Ex2} The phase space $C_{0} \times L^{p}(\rho ,X)$.
\rm Let  $1 \leq p < \infty $ and let  $\rho : (-\infty ,-r] \to
\real$ be a non-negative measurable function satisfying the conditions
(g-5)-(g-6) in the termino\-logy of Hino et al. \cite[p.15]{Hi}. Briefly, this means
that $ \rho $ is locally integrable and there exists a nonnegative  locally bounded function
$\gamma$ in  $(-\infty ,0]$ such that $
\rho (\xi +\theta )\leq \gamma (\xi )\rho (\theta )$ for all
$\xi \leq 0$ and $\theta \in (-\infty ,0] \setminus N_{\xi}$, where $
N_{\xi }\subset (-\infty , 0]$ is a set whose
Lebesgue measure is zero. Denote by
$ \mathcal{B} = C_{0} \times L^{p}(\rho, X) $ the set of
all classes of functions $ \varphi: (- \infty, 0] \to X $ such
that $ \varphi $  is Bochner measurable, and  the function
 $ \theta \mapsto\rho(\theta) \| \varphi(\theta) \|^{p} $ is Lebesgue integrable
on $ (- \infty, 0) $. The norm on $
C_{0}\times L^{p} (\rho, X) $ is defined by
\[
\|\vp \|_{\bb} =  \|\varphi (0) \| + \left(
\int_{- \infty}^{0} \rho(\theta) \|\varphi (\theta)\|^{p} d \theta \right)^{\frac{1}{p}}.
\]
The space $ \mathcal{B} = C_{0} \times L^{p}(\rho, X) $ is a phase space.
Moreover, when  $ p =  2$, we can  choose $ H =  1$,
$\; K(t)= 1+ \left( \int_{-t}^{0} \rho(\theta) d \theta \right)^{\frac{1}{2}} $ and
$M(t)=\gamma(-t)^{\frac{1}{2}} $ for $t\geq 0$ (see \cite[Theorem 1.3.8]{Hi}).
\end{example}

We also need some properties  of measure of
non-compactness. For this reason, we next recall a few properties
of this concept.
For general  information the reader can see \cite{AKPRS,BanGoe,ChK,heinz,KOZ}. In
this paper, we use the notion of Hausdorff measure of
noncompactness on the corresponding working space.
To establish some properties, in what follows we denote by $\chi$ the Hausdorff measure of noncompactness in $X$,
and by $\beta$ the
Hausdorff measure of noncompactness in a space of continuous functions with
values in $X$.

\begin{lema} \label{L2.4} {\rm(\cite[Corollary~2.10]{HPP})} Let $W \subseteq C(I, X)$ be a bounded set.
Then there exists a countable set $W_{0} \subseteq W$ such that
$\beta(W_{0}) = \beta(W)$.
\end{lema}

A set $W \subseteq L^{1}(I, X)$ is said to be uniformly integrable if there exists
a positive function  $\mu \in L^{1}(I)$ such that
$\|w(t) \| \leq \mu(t)$ a.e. for  $t \in I$ and all $w \in W$.

\begin{lema} \label{L2.5} {\rm(\cite[Lemma~2.13]{GH})} Let  $G : I \to \lx$ be a strongly continuous
operator valued map such that $G$ is  continuous for the norm of operators on $(0, a]$, and
let $\La : L^{1}(I, X)$  $\to C(I, X)$ be the map defined by
\[
\La(u)(t) = \int_{0}^{t} G(t -s) u(s) d s.
\]
Let $W \subset L^{1}(I, X)$ be a uniformly integrable set.
Assume that there is a positive function $q \in L^{1}(I)$ such that $\chi(W(t)) \leq q(t)$ for a.e.  $t \in I$.
Then
\[
\beta(\La(W)) \leq 2 \sup_{0 \leq t \leq a} \|G(t)\| \int_{0}^{a} q(t) d t.
\]
\end{lema}

We also need to relate the measure of noncompactness to the axioms of the phase space. This problem was addressed in \cite[Theorem~2.1]{Shin} for the Kuratowski measure of noncompactness; an analogous result holds for the Hausdorff measure of noncompactness. 

In what follows, any continuous function $u :[0, a] \to X$ with $u(0) = \varphi(0)$ for a fixed $\varphi \in \mathcal{B}$ is implicitly identified with its extension to $(-\infty, a]$ given by $u(\theta) = \varphi(\theta)$ for $\theta \leq 0$. Using axiom \textbf{(A)-(iii)} of the phase space, we obtain the following estimate.

\begin{lema} \label{L2.6} {\rm(\cite[Lemma~1.5]{Shin1})} Let $\vp \in \bb$, and let $U \subset C([0, a], X)$ be a bounded set such that $u(0) = \vp(0)$ for all
$u \in U$. Then
\[
\lambda(\{u_{a} : u \in U\}) \leq K(a) \beta(U),
\]
where $\lambda$ denotes the Hausdorff measure of noncompactness in the space $\bb$.
\end{lema}

\begin{de}
	A bounded continuous mapping $F:X\to X$ is said to be $\chi$-condensing if for every bounded subset $Q$ of $X$ for which $\chi(Q)>0$ we have
	\[
	\chi(F(Q))<\chi(Q).
	\]
\end{de}
In the next, we recall a fixed point principle due to B. N.  Sadovskii \cite{Sad}.

\begin{theorem}  \label{T2.1} Let $M$ be a bounded, closed and convex subset of $X$, and let $F : M \to M$ be a $\chi$-condensing  map.
Then $Fix(F)=\{x\in M: F(x)=x \}$ is a
nonempty compact set.
\end{theorem}

 A function $f : [a, b] \to  X$ is said to be
regulated (see \cite[Section VII.6]{D})  if for every $t \in [a, b)$ the right-sided limit $\lim_{s \to t^{+}} f(s) = f(t^{+})$
exists in $X$, and for every $t \in (a, b]$ the left-sided limit $\lim_{s \to t^{-}} f(s) = f(t^{-})$ exists in $X$.
It is known  (\cite{Honig}) that the set of discontinuities of a regulated
function is at most countable. We denote by $G([a, b], X)$ the space consisting of all regulated functions from
$[a, b]$ into $X$. The space  $G([a, b], X)$ endowed with the norm of uniform
convergence is a Banach space. Moreover,
every  $f \in G([a, b], X)$ is the limit of a sequence of step functions with respect to the norm of uniform convergence on $[a, b]$
(\cite{D}).

\setcounter{equation}{0}
\setcounter{remark}{0}
\setcounter{pr}{0}
\setcounter{de}{0}
\setcounter{theorem}{0}
\setcounter{lema}{0}
\setcounter{co}{0}
\setcounter{example}{0}

\section{Existence of mild solutions} \label{existence}
Throughout the rest of this work, we assume that $\bb$ is a phase space,
$r: [0, a] \to \real $ is a regulated function such that $r(s) \leq s$, for
all $0 \leq s \leq a$, and $f : [0, a]  \times \bb  \to X $ is a function that satisfies
the following Carath\'eodory conditions:
\begin{description}
\item[(Car1)] The function $f(\cdot, \vp)$ is measurable on
$[0, a]$ for each $\vp \in \bb$, and the function $f(t, \cdot)$ is continuous on
$\bb$ for almost all $t \in [0, a]$.

\medskip

\item[(Car2)] For each $R \geq 0$, there exists a positive function
$\phi_{R} \in L^{1}([0, a])$ such that $\|f(t, \vp)\| \leq \phi_{R}(t)$ a.e. for all
$\vp \in \bb$ with $\|\vp\|_{\bb} \leq R$.
\end{description}

\medskip

In the rest of this work, we denote $N =  \sup_{0 \leq t \leq a} \|T(t)\|$.
Furthermore, in what follows we denote by $\|\cdot\|_{\infty}$ the norm of the uniform convergence in any
space $C([0, b],X)$  for $0 < b \leq a$,  $K_{a} = \max_{0 \leq t \leq a} K(t)$ and $M_{a} = \sup_{0 \leq t \leq a} M(t)$.

\begin{remark}\rm
 Consider that the initial condition \eqref{equ1.3} belongs to $\bb$, i.e., $x_0=\vp\in\bb$. Since the regulated function $r$ can take negative values, it is not guaranteed that $x_{r(t)}$ belongs to the phase space $\bb$, for all $t\in[0,a]$. Consequently, the term $f(t,x_{r(t)})$ in \eqref{equ1.2} is not necessarily well-defined. This specific issue was studied in \cite{HMH}. Hence, it is necessary to establish an appropriate space for the initial conditions, a definition that we shall provide in what follows.
\end{remark}

\begin{de}\rm
Let $\tau = \inf_{0 \leq t \leq a} (r(t) - t)$. We define the vector space $\mathcal{B}_{\tau}$ consisting of all functions $\varphi \in \mathcal{B}$ that satisfy the following conditions:
\begin{itemize}
    \item $\varphi_{s} \in \mathcal{B}$, for all $s \in [\tau, 0]$,

    \medskip
    
    \item The mapping $s \mapsto \varphi_{s}$ is continuous from $[\tau, 0]$ into $\mathcal{B}$.
\end{itemize}

\end{de}

\begin{lema} \label{R3.1} {\rm (\cite[Lemma~3.1]{HMH})}
    The vector space $(\mathcal{B}_{\tau},\|\cdot\|_{\tau})$ is complete, where
\[
\|\varphi\|_{\tau} = \sup_{\tau \leq s \leq 0} \|\varphi_{s}\|_{\mathcal{B}}, \quad \forall \varphi \in \mathcal{B}_{\tau}.
\] 
\end{lema}

From Remark~\ref{R2.1}, the space $\mathcal{B}_\tau$ contains $C_{00}$. Moreover, for every $\phi \in C_b((-\infty,0],X))$ and every $\tau \le s \le 0$, we have $\phi_s \in C_b((-\infty,0],X) \subset \mathcal{B}$. Since the map $s \mapsto \phi_s$ is continuous by axiom (A)-(i), it follows that $C_b((-\infty,0],X)$ is continuously embedded in $\mathcal{B}_\tau$. That is, there exists a constant $Q>0$ such that $\|\vp\|_{\tau}\leq Q\|\vp\|_{\infty}$, for all $\vp\in C_b([-\infty,0),X)$.

The following properties will be useful in the sequel.

\begin{lema} \label{L3.1}
Assume that $x: (-\infty, a] \to X$ is a continuous function on $[0, a]$ such that  $x_0 = \varphi \in \mathcal{B_{\tau}}$.
Then the function $ [0,a] \to \bb_{\tau}$, $t \mapsto x_{t}$, is continuous.
\end{lema}
\begin{proof} It follows from axiom \textbf{(A)-(i)} that the function $ [0,a] \to \bb$, $t \mapsto x_{t}$, is continuous and, therefore,
uniformly continuous on $[0, a]$.  Moreover, for $\tau \leq \xi \leq 0$, we have
\[
[x_{t}]_{\xi} = \left\{ \begin{array}{rcl} \vp_{t+ \xi}, & & \tau \leq \xi \leq -t, \\
x_{t + \xi}, & &   -t\leq \xi \leq 0. \end{array} \right.
\]
To study the continuity of function $t \mapsto x_{t}$, we take $h \geq 0$, and estimate
\[
\|x_{t +h} - x_{t} \|_{\tau}   =   \sup_{\tau \leq \xi \leq 0} \|[x_{t +h} - x_{t}]_{\xi} \|_{\bb}
 =  \displaystyle  \max \{ c_{1}, c_{2}, c_{3} \},
\]
where we set ${c_{1} = \sup_{-t \leq \xi \leq 0} \|x_{t +h + \xi} - x_{t + \xi}\|_{\bb}}$, 
 ${ c_{2} = \sup_{-t -h \leq \xi \leq -t} \|x_{t +h + \xi} - \vp_{t + \xi}\|_{\bb}}$, and
  $c_{3}=\sup_{\tau \leq \xi \leq -t-h} \|\vp_{t +h + \xi} - \vp_{t + \xi}\|_{\bb}$.
In addition, for $-t -h \leq \xi \leq -t$,
\[
\|x_{t +h + \xi} - \vp_{t + \xi}\|_{\bb} \leq \|x_{t +h + \xi} - \vp\|_{\bb} + \|\vp - \vp_{t + \xi}\|_{\bb}.
\]
Combining these estimates, we infer that $\|x_{t +h} - x_{t} \|_{\tau} \to 0$ as $h \to 0$, which completes the proof.

\end{proof}

\begin{lema} \label{L1} {\rm (\cite[Lemma~3.2]{HMH})}
Assume that $x: (-\infty, a] \to X$ is a continuous function on $[0, a]$ such that  $x_0 = \varphi \in \mathcal{B_{\tau}}$.
Then the function $u : [0,a] \to X$ given by
\[
u(t) = f\left(t, x_{r(t)}\right), \; 0 \leq t \leq a,
\]
is measurable in the Bochner sense.
\end{lema}

We recall the concept of mild solution for problem (\ref{equ1.2})-(\ref{equ1.3}) introduced in \cite[Definition~3.4]{HMH}.

\begin{de} \label{D1} A   function  $x : (- \infty, a] \to X$ is a mild solution of
problem \textup{(\ref{equ1.2})-(\ref{equ1.3})} if $x$ is continuous on $[0, a]$,
$x_{0} = \vp \in \bb_{\tau}$, and the integral equation
\begin{equation}
x(t)  = T(t) \vp(0) +  \int_{0}^{t} T(t -s) f(s,x_{r(s)})  d s, \;\; t \in [0, a], \label{equ3.2}
\end{equation}
is verified.
\end{de}

Next, we recall an existence and uniqueness theorem established in \cite{HMH}. To this end, we introduce the following global Lipschitz condition.

\begin{description}
\item[(H1)]
There exists a positive function $\eta \in L^{1}([0,a])$ such that
\[
\|f(t, \psi_{1}) - f(t,  \psi_{2}) \| \leq \eta(t) \|\psi_{1}  - \psi_{2} \|_{\bb},
\]
for $ \psi_{1}, \psi_{2} \in \bb$ and $t \in [0, a]$.
\end{description}

\begin{theorem}\label{T1} {\rm (\cite[Theorem~3.5]{HMH})}
Let $\vp \in \bb_{\tau}$. Assume that condition {\rm \bf (H1)} is satisfied.
Then there exists a unique mild solution of problem \textup{(\ref{equ1.2})-(\ref{equ1.3})}.
\end{theorem}

It is worth noting that the global condition {\rm \bf (H1)} can be straightforwardly relaxed to a Lipschitz condition on bounded sets, provided that a certain asymptotic behavior is assumed for the function $f$. By a standard argument, a global existence and uniqueness result can still be established. To this end, we introduce the following conditions:

\begin{description}
\item[(H2)] For each $R > 0$, there exists a positive function $\eta_{R}
\in L^{1}([0,a])$ such that
\[
\|f(t, \vp^{1}) - f(t, \vp^{2}) \| \leq \eta_{R}(t)  \|\vp^{1}- \vp^{2} \|_{\bb},
\]
for $\vp^{i} \in \bb_{\tau}$, $\|\vp^{i}\|_{\bb} \leq R$,
$i = 1, 2$, and $t \in [0, a]$.

\medskip

\item[(H3)] There exists a positive function $m \in L^{1}([0,a])$ and a
continuous nondecreasing function
$\Phi: [0, \infty) \to [0, \infty)$ such that
\[
\|f(t,  \psi)\| \leq m(t) \Phi( \|\psi\|_{\bb}),
\]
for all $ \psi \in \bb$ and a.e. $t \in [0, a]$.
\end{description}

\begin{theorem} \label{T7}
Let $\vp \in \bb_{\tau}$. Assume that conditions {\rm \bf (H2)} and {\rm \bf (H3)} are  satisfied.
Assume further that there exists $R > 0$ such that
\begin{equation}
\|\vp\|_{\tau}+ N \|\vp(0)\| + N \int_{0}^{a} m(s) d s \, \Phi(K_{a} R + M_{a} \|\vp\|_{\tau}) \leq R. \label{equ3.8}
\end{equation}
Then there exists a unique mild solution for problem \eqref{equ1.2}-\eqref{equ1.3}.
\end{theorem}

\begin{proof}
    Consider the Banach space 
    \[
    Y=\{x\colon(-\infty,a]\to X: \text{ $x_0\in\bb_{\tau}$ and $x|_{[0,a]}$ is continuous}\}
    \] 
    endowed with the norm $\|x\|_{Y}=\|x_0\|_{\tau}+\sup_{\theta\in[0,a]}\|x(\theta)\|$.
    
    Now, let $\vp \in \bb_{\tau}$ be fixed and consider the closed convex subset $Y_{\vp}=\{x\in Y: x_0=\vp\}$ of $Y$. Define the operator $\Gamma\colon Y_{\vp}\to Y_{\vp}$ by
   \begin{equation}
\label{equ3.1}
\Gamma(x)(t) = \left\{
\begin{aligned}
& T(t)\varphi(0) + \int_{0}^{t} T(t-s)f(s, x_{r(s)}) \, ds, && 0 \leq t \leq a, \\
& \varphi(t), && t \leq 0.
\end{aligned}
\right.
\end{equation}
For $u\in Y_{\vp}$  with $\|u\|_Y \leq R$, we have
\[
\|\Gamma (u) (t)\|   \leq  N \|\vp(0)\| + N \int_{0}^{t}  m(s) d s \Phi (K_{a} R + M_{a} \|\vp\|_{\tau}) d s, \qquad \forall t \in [0, a].
\]
Combining this estimate with \eqref{equ3.8}, we obtain that
$\Gamma (B_{R}(0,Y_{\vp})) \subseteq  B_{R}(0,Y_{\vp})$.
We can argue as in the proof of \cite[Theorem~3.5]{HMH} but with  the map $\Gamma$ on $B_{R}(0,Y_{\vp})$ and $\eta_{R}$ instead of
$\eta$ to conclude that $\Gamma$ is a contraction on $B_{R}(0,Y_{\vp})$. Then, the unique fixed point of $\Gamma$ is the mild solution of problem \eqref{equ1.2}-\eqref{equ1.3}.

\end{proof}

Next, we will study the existence of solutions for problem \eqref{equ1.2}-\eqref{equ1.3} for immediately norm continuous semigroup (see \cite[Definition~II.4.17]{EN}). In order to establish our results quite generally, we will use the theory of measures of non-compactness and the following condition:

\begin{description}
\item[(H4)]
There exists a positive  function $\mu \in L^{1}([0, a]) $  such that
\[
\chi(f(t, B)) \leq \mu(t) \chi(B), \; a.e. \; t \geq 0,
\]
for all  bounded set $B \subseteq \bb$, and where $\chi$ and $\lambda$ denotes the Hausdorff measure of non-compactness both in $X$ and $\bb$, respectively.
\end{description}

\begin{remark}\label{rem:gamma-domain}
\rm
Notice that the operator $\Gamma$ defined in Theorem 3.2 acts on the entire interval $(-\infty, a]$ to properly accommodate the infinite delay and the initial history $\varphi$. To justify that the subsequent proofs of Theorem 3.3 and Corollary 3.2 can be conducted solely on the interval $[0, a]$, consider the projection map $p_2: Y_{\varphi} \to C([0,a], X)$ given by $p_2(x) = x|_{[0,a]}$. For any $x, y \in Y_{\varphi}$, since $x_0 = y_0 = \varphi$, it follows from the definition of the norm $\|\cdot\|_Y$ that
\[
\|x - y\|_Y = \|x_0 - y_0\|_{\tau} + \sup_{t \in [0,a]} \|x(t) - y(t)\| = \|p_2(x) - p_2(y)\|_{\infty}.
\]
Thus, the map $p_2$ restricted to $Y_{\varphi}$ is an isometry. Note that the range of $p_2$ is exactly the closed convex subset $C_{\varphi(0)}([0,a], X) := \{u \in C([0,a], X) : u(0) = \varphi(0)\}$, meaning that $p_2: Y_{\varphi} \to C_{\varphi(0)}([0,a], X)$ is a bijective isometry. Consequently, a subset $\Omega \subseteq Y_{\varphi}$ is relatively compact in $Y_{\vp}$ if and only if $p_2(\Omega)$ is relatively compact in $C([0,a], X)$. Furthermore, the Hausdorff measures of noncompactness in $Y_{\vp}$ and $C([0,a],X)$, denoted by $\beta_{Y_{\vp}}$ and $\beta$ respectively, satisfy $\beta_{Y_{\vp}}(\Omega) = \beta(p_2(\Omega))$ and $\beta_{Y_{\vp}}(\Gamma(\Omega)) = \beta(p_2(\Gamma(\Omega)))$. Hence, without loss of generality, we can identify $Y_{\varphi}$ with $C_{\varphi(0)}([0,a], X)$.
\end{remark}

\begin{theorem} \label{T3.1} Let $(T(t))_{t \geq 0}$ be an immediately norm continuous semigroup and $\vp \in \bb_{\tau}$.
Assume further conditions {\rm \bf (H3)}, {\rm \bf (H4)} and  \eqref{equ3.8} are fulfilled, and
\[
2 N  K_{a} \int_{0}^{a}  \mu(s)   d s < 1.
\]
Then there exists a mild solution  for problem \eqref{equ1.2}-\eqref{equ1.3}.
\end{theorem}
\begin{proof}  We consider $\Gamma$ defined by \eqref{equ3.1}. Using (Car1)-(Car2) and {\rm \bf (H3)} we obtain
that $\Gamma$ is a continuous map. Moreover, proceeding as in the proof of Theorem~\ref{T7}
 we can prove that  that $\Gamma (B_{R}(0, Y_{\varphi})) \subseteq B_{R}(0, Y_{\varphi})$.

Next we show that   $\Gamma$ is $\beta$-condensing on $B_{R}(0, Y_{\varphi})$.
Let $\Om \subset B_{R}(0, Y_{\varphi}) $. It follows from Lemma~\ref{L2.4} that there exists
a sequence $(v^{k})_{k}$ in $\Gamma(\Om)$ such that  $\beta(\Gamma(\Om)) =
\beta(\{v^{k}: k \in \nat \})$.

Using \eqref{equ3.1} we can write
\begin{equation}
v^{k}(t)   =  T(t)\vp(0) + \int_{0}^{t} T(t-s) f(s, u^{k}_{r(s)}) d s, \; 0 \leq t \leq a,
\label{equ7}
\end{equation}
where  $u^{k} \in \Om$ and $v^{k} = \Gamma(u^{k})$. Let $U = \{u^{k} : k \in \nat \} \subset C([0,a], X)$ and
$w^{k}$ be the function given by
$w^{k}(s) = f(s, u^{k}_{r(s)}) $ for $0 \leq s \leq a$.
From the axiom\textbf{(A)-(iii)}, it follows that
\begin{eqnarray*}
\|f(s, u^{k}_{r(s)})\| & \leq & m(s) \Phi(\|u^{k}_{r(s)}\|_{\bb}) \\
& \leq & m(s) \Phi(K_{a} R + M_{a} \|\vp\|_{\tau})
\end{eqnarray*}
for $0 \leq s \leq a$. Hence, the set $W = \{w^{k}: k \in \nat \}$ is uniformly integrable.
Furthermore, if $r(t) < 0$, then $u^{k}_{r(t)} = \vp_{r(t)}$ for all $k \in \nat$. Hence
\[
\chi(\{u^{k}_{r(t)} : k \in \nat \}) = 0,
\]
and
\[
\chi(W(t)) = 0.
\]
In addition, if $r(t) \geq  0$, using Lemma~\ref{L2.6}, we have
\[
\chi(\{u^{k}_{r(t)} : k \in \nat \}) \leq  K(r(t)) \beta(U) \leq  K_{a} \beta(\Om)
\]
and, it follows from {\rm \bf (H4)} that
\[
\chi(W(t)) \leq \mu(t) \chi(\{u^{k}_{r(t)} : k \in \nat \}) \leq   K_{a} \mu(t) \beta(\Om).
\]
Combining these observations, and using Lemma~\ref{L2.5}, we obtain
\[
\beta(\{v^{k}(\cdot): k \in \nat\})  \leq
2 N  K_{a} \int_{0}^{a}  \mu(s)   d s  \beta(\Om),
\]
which  implies that $\Gamma $ is a $\beta$-condensing map, and applying Theorem~\ref{T2.1},
we infer the existence of a mild solution  of problem \eqref{equ1.2}-\eqref{equ1.3}.

\end{proof}

The following consequence is immediate.
\begin{co} \label{C3.1}  Let $(T(t))_{t \geq 0}$ be an immediately norm continuous semigroup and $\vp \in \bb_{\tau}$.
Assume further that conditions {\rm \bf (H3)} and  \eqref{equ3.8} are fulfilled, and
for every $0 \leq t \leq a$ and $\rho \geq 0$ the set $\{f(t, \psi) : \psi \in \bb, \; \|\psi\|_{\bb} \leq \rho \}$
is relatively compact.
Then there exists a mild solution  for problem \eqref{equ1.2}-\eqref{equ1.3}.
\end{co}

Many strongly continuous semigroups that arise in applications have  a compactness property. In what follows we use the concept of compact semigroup (see \cite[Definition~II.4.23]{EN}). It follows from \cite[Lemma~II.4.22]{EN} that a compact semigroup is also immediately norm continuous.

\begin{co} \label{C3.2} Let $(T(t))_{t \geq 0}$ be a compact semigroup and $\vp \in \bb_{\tau}$.
Assume further that conditions {\rm \bf (H3)} and  \eqref{equ3.8} are fulfilled.
Then there exists a mild solution  of problem \eqref{equ1.2}-\eqref{equ1.3}.
\end{co}
\begin{proof} We proceed as in the proof of Theorem~\ref{T3.1}. It only remains to
establish that the set $\Gamma(B_{R}(0, Y_{\varphi})) $ is  relatively compact. Initially, we will show that $\Gamma(B_{R}(0, Y_{\varphi}))(t) $ is a relatively
compact set in $X$ for all $t \in [0, a]$. The case $t = 0$ is immediate. Assume that $t > 0$.
\[
\Gamma(B_{R}(0, Y_{\varphi}))(t) = \{ v(t) = \Gamma(u)(t): u \in B_{R}(0, Y_{\varphi}) \}.
\]
Let $0 < \ve < t$. For each $u \in B_{R}(0, Y_{\varphi})$, we define
\begin{eqnarray*}
v^{\ve}(t) & = & T(t) \vp(0) + \int_{0}^{t-\ve} T(t -s) f(s, u_{r(s)}) d s \\
& = & T(t) \vp(0) + T(\ve) \int_{0}^{t-\ve} T(t - \ve -s) f(s, u_{r(s)}) d s.
\end{eqnarray*}
Using {\rm \bf (H3)} and the fact that $T(\ve)$ is a compact linear operator, we obtain that the set
$\{v^{\ve}(t) : v = \Gamma(u), \; u \in B_{R}(0, Y_{\varphi}) \}$ is relatively compact. Moreover, since
\[
v(t) -  v^{\ve}(t) = \int_{t-\ve}^{t} T(t  -s) f(s, u_{r(s)}) d s \to 0, \; \ve \to 0,
\]
uniformly for $u \in B_{R}(0, Y_{\varphi})$, we conclude that $\Gamma(B_{R}(0, Y_{\varphi}))(t) $ is a relatively compact set.

On the other hand, the set $\Gamma(B_{R}(0, Y_{\varphi})) $ is equicontinuous. In fact, let $v = \Gamma(u)$ for $u \in B_{R}(0, Y_{\varphi})$.
We have that
\begin{align*}
\lefteqn{v(t + h) - v(t)  =  T(t) [T(h) \vp(0) - \vp(0)] + \int_{0}^{t + h} T(t + h -s) f(s, u_{r(s)}) d s} \\
& \quad - \int_{0}^{t} T(t  -s) f(s, u_{r(s)}) d s \\
& =  T(t) [T(h) \vp(0) - \vp(0)] + (T(h) - I) \int_{0}^{t} T(t -s) f(s, u_{r(s)}) d s \\
& \quad + \int_{t}^{t +h} T(t + h  -s) f(s, u_{r(s)}) d s \\
& =  T(t) [T(h) \vp(0) - \vp(0)] + (T(h) - I) [v(t) - T(t) \vp(0)]  +
\int_{t}^{t +h} T(t + h  -s) f(s, u_{r(s)}) d s \\
& =   (T(h) - I) v(t)  + \int_{t}^{t +h} T(t + h  -s) f(s, u_{r(s)}) d s.
\end{align*}
Since $T(t + h  -s) f(s, u_{r(s)})$ is uniformly bounded for $u \in B_{R}(0, Y_{\varphi})$ and the set
$\{v(t) : u \in B_{R}(0, Y_{\varphi}) \}$ is relatively compact, we obtain that
\[
v(t + h) - v(t) \longrightarrow 0, \; h \to 0,
\]
uniformly for $u \in B_{R}(0, Y_{\varphi})$. Combining these assertions, and applying the Arzel\'a-Ascoli theorem, we conclude that
$\Gamma(B_{R}(0, Y_{\varphi})) $ is relatively compact in $Y_{\varphi}$. We complete the proof arguing as in the proof of Theorem~\ref{T3.1}.
\end{proof}

\begin{remark}
    \rm Regarding uniqueness for compact semigroups of linear operators, we refer to \cite[Theorem~3.6]{HMH}.
\end{remark}

\setcounter{equation}{0}
\setcounter{remark}{0}
\setcounter{pr}{0}
\setcounter{de}{0}
\setcounter{theorem}{0}
\setcounter{lema}{0}
\setcounter{co}{0}
\setcounter{example}{0}

\section{Applications} \label{applications}

The purpose of this section is to establish the existence of solutions for the non-local diffusion system \eqref{equ1.4}--\eqref{equ1.6} by formulating it as an abstract Cauchy problem on the Hilbert space $X = L^{2}([0, \pi])$.

Let $A: D(A)\subset X \to X$ be the linear operator defined by $A z = z^{\prime\prime}$ on the domain
\[
D(A) =\{ z \in X : z^{\prime \prime} \in X, \, z(0) = z(\pi) = 0\}.
\]
It is well known that $A$ is the infinitesimal generator of an analytic semigroup $(T(t))_{t\ge0}$ on $X$. Moreover, the operator $A$ possesses a discrete spectrum with eigenvalues $\lambda_n = -n^{2}$ and corresponding normalized eigenfunctions $z_{n}(\xi)= (2/\pi)^{1/2} \sin(n \xi)$, for $n\in\mathbb{N}$. Since $\{z_n\}_{n\in \mathbb{N}}$ constitutes an orthonormal basis of $X$, the operator $A$ and its generated semigroup admit the spectral representations
\[
A z = \sum_{n= 1}^{\infty} - n^{2} \langle z, z_{n} \rangle z_{n}, \quad z \in D(A),
\]
and
\[
T(t) z = \sum_{n= 1}^{\infty} e^{- n^{2} t} \langle z, z_{n} \rangle z_{n}, \quad z \in X, \, t \ge 0.
\]
These series representations immediately imply that $(T(t))_{t \geq 0}$ is a compact and contractive semigroup on $X$; that is, $\|T(t)\| \leq 1$ for all $t \geq 0$.

To model \eqref{equ1.4}--\eqref{equ1.6} on the interval $[0, a]$, we adopt the standard identification $[\psi(\theta)](\xi)= \psi(\theta, \xi)$. Under the assumptions that the delay function $r\colon [0, a] \to \mathbb{R}$ is regulated with $r(t) \leq t$ on $[0, a]$, we define $f\colon [0,a] \times \mathcal{B} \to X$ by $f(t,\psi)=\Lambda(\psi)+h(t)$ where   $\Lambda\colon\mathcal{B} \to X$ is a linear operator defined by
\[
[\Lambda(\psi)](\xi) = \int_{-\infty}^{0} P(\theta) \psi(\theta, \xi) \, d \theta, \quad 0 \leq \xi \leq \pi,
\] and $h\colon [0, a] \to X$ is defined by $[h(t)](\xi) = h_{0}(t, \xi)$, where $h_0(t,\cdot) \in X$ integrable function. Consequently, problem \eqref{equ1.4}--\eqref{equ1.6} can be modeled in the abstract form \eqref{equ1.2}--\eqref{equ1.3} on the phase space $\mathcal{B}$.
\medskip

We consider two cases, depending on the space to which $\vp$ belongs:

\medskip

\noindent{\bf Case 1:} Assume that $\vp \in C_{g}^{0}(X)$, where  $g$ is a function that satisfies the conditions considered in Example~\ref{Ex1}.
We take as phase space $\bb = C_{g}^{0}(X)$. We assume that there exists a measurable function $Q: (- \infty, 0] \to (0, \infty)$ such that
${\displaystyle \int_{-\infty}^{0} P(\ta)^{2} Q(\ta)^{2}  d \ta < \infty}$  and
${\displaystyle \int_{-\infty}^{0} \frac{g(\ta)^{2}}{Q(\ta)^{2}}  d \ta < \infty}$.
For $\psi_1,\psi_2 \in \bb$ and $t \in [0,a]$, we have
\begin{eqnarray*}
\|f(t,\psi_1)-f(t,\psi_2)\|^2&=&
\|\La(\psi_1-\psi_2)\|^{2}\\ & = & \int_{0}^{\pi} \left| \int_{-\infty}^{0} P(\ta) (\psi_1(\ta, \xi)-\psi_2(\ta, \xi)) d \ta \right|^{2} d \xi \\
& = & \int_{0}^{\pi} \left| \int_{-\infty}^{0} P(\ta) Q(\ta) \frac{1}{Q(\ta)} (\psi_1(\ta, \xi)-\psi_2(\ta, \xi)) d \ta \right|^{2} d \xi \\
& \leq  & \int_{0}^{\pi} \left(\int_{-\infty}^{0} P(\ta)^{2} Q(\ta)^{2}  d \ta \right)  \int_{-\infty}^{0} \frac{1}{Q(\ta)^{2}}
|(\psi_1(\ta, \xi)-\psi_2(\ta, \xi))|^{2} d \ta  d \xi \\
& = & \left(\int_{-\infty}^{0} P(\ta)^{2} Q(\ta)^{2}  d \ta \right)   \int_{-\infty}^{0} \int_{0}^{\pi} \frac{1}{Q(\ta)^{2}}
|(\psi_1(\ta, \xi)-\psi_2(\ta, \xi))|^{2}   d \xi d \ta \\
%& = & \left(\int_{-\infty}^{0} P(\ta)^{2} Q(\ta)^{2}  d \ta \right)   \int_{-\infty}^{0} \frac{g(\ta)^{2}}{Q(\ta)^{2}} \frac{1}{g(\ta)^{2}}
%\int_{0}^{\pi}  |(\psi_1(\ta, \xi)-\psi_2(\ta, \xi))|^{2}   d \xi d \ta \\
& \leq  & \left(\int_{-\infty}^{0} P(\ta)^{2} Q(\ta)^{2}  d \ta \right)   \int_{-\infty}^{0} \frac{g(\ta)^{2}}{Q(\ta)^{2}}  d \ta \;
\|\psi_1-\psi_2\|_{\bb}^{2},
\end{eqnarray*}
which implies that $f$ satisfies the condition \textbf{(H1)}.

Thus, a direct application of Theorem~\ref{T1} yields the following result.

\begin{theorem} \label{C5.1}  Under the above conditions, assume in further that
 $\vp_{\tau} \in C_{g}^{0}(X)$. Then  problem \eqref{equ1.4}-\eqref{equ1.6} has a
mild solution $w \in C([0, a],  L^{2}([0, \pi]))$.
\end{theorem}
\begin{proof} It only remains to prove that $\vp_{s} \in \bb$ for all $\tau \leq s \leq 0$, and that the function
$[\tau, 0] \to \bb$, $s \mapsto  \vp_{s}$ is continuous. For $\tau \leq s \leq 0$, $\ta \leq 0$ and $\xi = \ta +s - \tau$, we have
\begin{eqnarray*}
\frac{\|\vp_{s}(\ta)\|}{g(\ta)} & = & \frac{\|\vp(s +\ta)\|}{g(\ta)} \\
& = & \frac{\|\vp(s + \tau - s + \xi)\|}{g(\tau - s + \xi)} \\
& \leq  & \frac{\|\vp( \tau  + \xi)\|}{g( \xi)} \\
& \leq  & \frac{\|\vp_{\tau} (\xi)\|}{g( \xi)},
\end{eqnarray*}
which implies that ${\displaystyle \frac{\|\vp_{s}(\ta)\|}{g(\ta)} \to 0}$ as $\ta \to - \infty$. Therefore, $\vp_{s} \in C_{g}^{0}(X)$,
and this shows the first assertion. To establish the second assertion, let $\ve > 0$ and  $\tau \leq s, s^{\prime} \leq 0$.
Proceeding as above, there exists $\ta_{0} < 0$ such that
${\displaystyle \frac{1}{g(\ta)} \|\vp(s^{\prime} + \ta) \| \leq \ve} $ and
${\displaystyle \frac{1}{g(\ta)} \| \vp(s + \ta) \| \leq \ve}$ for all $\ta \leq \ta_{0}$. Hence,
\begin{eqnarray*}
\|\vp_{s^{\prime}} - \vp_{s} \|_{\bb} & = & \sup_{\ta \leq 0} \frac{1}{g(\ta)} \|\vp(s^{\prime} + \ta)  - \vp(s + \ta) \| \\
& \leq  & \max \Big\{ 2 \ve, \sup_{\ta_{0} \leq \ta \leq 0} \frac{1}{g(\ta)} \|\vp(s^{\prime} + \ta)  - \vp(s + \ta) \| \Big\}.
\end{eqnarray*}
Since $\vp$ is uniformly continuous on bounded intervals, we can take $|s^{\prime} - s|$ sufficiently small such that
${\displaystyle \sup_{\ta_{0} \leq \ta \leq 0} \frac{1}{g(\ta)} \|\vp(s^{\prime} + \ta)  - \vp(s + \ta) \| \leq \ve }$. Combining
these assertions, we conclude that  $\|\vp_{s^{\prime}} - \vp_{s} \|_{\bb} \leq 2 \ve$. This completes the proof of the second assertion.

\end{proof}

\begin{remark} \label{Ex4}
\rm Under conditions (g-1) and (g-2) established in Example~\ref{Ex2} we can affirm that the phase space  $C_{g}^{0}(X)$  satisfies axiom \textbf{(C-2)} (\cite[Theorem~1.3.2]{Hi}).
In similar way, if $\rho$ satisfies the conditions (g-5) and (g-6) given  in Example~\ref{Ex2}
and  $\int_{-\infty}^{0} \rho(\ta) d \ta < \infty$, then the
phase space $C_{0} \times L^{p}(\rho ,X)$ satisfies axiom \textbf{(C-2)} (\cite[Theorem~1.3.8]{Hi}).
\end{remark}

\medskip

\noindent{\bf Case 2:} Assume that $\vp \in  C_{0} \times L^{2}(\rho ,X)$, where  $\rho$ is a function that satisfies the conditions
considered in Example~\ref{Ex2} and Remark~\ref{Ex4}.
We take as phase space $\bb = C_{0} \times L^{2}(\rho ,X)$.
Assume that
${\displaystyle \int_{-\infty}^{0} \frac{|P(\ta)|^{2}}{\rho(\ta)}  d \ta < \infty}$.
For $\psi_1,\psi_2 \in \bb$ and $t \in [0,a]$, we have
\begin{eqnarray*}
\|f(t,\psi_1)-f(t,\psi_2)&=&\|\La(\psi_1-\psi_2)\|^{2}\\ & = & \int_{0}^{\pi} \left| \int_{-\infty}^{0} P(\ta) (\psi_1(\theta,\xi)-\psi_2(\theta,\xi)) d \ta \right|^{2} d \xi \\
& \leq  & \int_{0}^{\pi} \left(\int_{-\infty}^{0} \frac{|P(\ta)|^{2}}{\rho(\ta)} d \ta \right)  \int_{-\infty}^{0} \rho(\ta)
|(\psi_1(\theta,\xi)-\psi_2(\theta,\xi))|^{2} d \ta  d \xi \\
& = & \left(\int_{-\infty}^{0} \frac{|P(\ta)|^{2}}{\rho(\ta)}  d \ta \right)   \int_{-\infty}^{0} \rho(\ta) \int_{0}^{\pi}
|(\psi_1(\theta,\xi)-\psi_2(\theta,\xi))|^{2}   d \xi d \ta \\
& \leq  & \left(\int_{-\infty}^{0} \frac{|P(\ta)|^{2}}{\rho(\ta)}  d \ta \right)   \|\psi_1-\psi_2\|_{\bb}^{2}
\end{eqnarray*}
which implies that $f$ satisfies the condition \textbf{(H1)}.
Thus, a direct application of Theorem~\ref{T1} yields the following result.

\begin{theorem} \label{C5.2}  Under the above conditions, assume in further that $\vp_{\tau} \in C_{0} \times L^{2}(\rho ,X)$ and
 $\vp\colon[\tau, 0] \to X$ is continuous. Then  problem \eqref{equ1.4}-\eqref{equ1.6} has a
mild solution $w \in C([0, a],  L^{2}([0, \pi]))$.
\end{theorem}
\begin{proof} It only remains to prove that $\vp_{s} \in \bb$ for all $\tau \leq s \leq 0$, and that the function
$[\tau, 0] \to \bb$, $s \mapsto  \vp_{s}$ is continuous. For $\tau \leq s \leq 0$, we have
\begin{eqnarray*}
\lefteqn{\left(\int_{-\infty}^{0} \rho(\ta) \|\vp(s + \ta)\|^{2} d \ta \right)^{1/2}  =  \left(\int_{-\infty}^{\tau -s} \rho(\ta) \|\vp(s + \ta)\|^{2} d \ta
+ \int_{\tau -s}^{0} \rho(\ta) \|\vp(s + \ta)\|^{2} d \ta \right)^{1/2}} \\
& \leq  & \left(\int_{-\infty}^{\tau -s} \rho(\ta) \|\vp(s + \ta)\|^{2} d \ta \right)^{1/2}
+ \left( \int_{\tau -s}^{0} \rho(\ta) \|\vp(s + \ta)\|^{2} d \ta \right)^{1/2} \\
& \leq   & \left(\int_{-\infty}^{0} \rho(\xi + \tau -s) \|\vp(\tau + \xi)\|^{2} d \xi \right)^{1/2}
+ \left( \int_{\tau -s}^{0} \rho(\ta)  d \ta \right)^{1/2} \max_{\tau \leq \xi \leq 0} \|\vp(\xi)\| \\
& \leq   & \gamma(\tau -s)^{1/2} \left(\int_{-\infty}^{0} \rho(\xi) \|\vp(\tau + \xi)\|^{2} d \xi \right)^{1/2}
+ \left( \int_{\tau -s}^{0} \rho(\ta)  d \ta \right)^{1/2} \max_{\tau \leq \xi \leq 0} \|\vp(\xi)\| \\
& \leq   & \gamma(\tau -s)^{1/2} \|\vp_{\tau}\|_{\bb}
+ \left( \int_{\tau -s}^{0} \rho(\ta)  d \ta \right)^{1/2} \max_{\tau \leq \xi \leq 0} \|\vp(\xi)\|,
\end{eqnarray*}
which implies that ${\displaystyle \|\vp_{s}\|_{\bb} < \infty}$. Therefore, $\vp_{s} \in C_{0} \times L^{2}(\rho ,X)$,
which shows the first assertion. To establish the second assertion, let $\ve > 0$ and  $\tau \leq s, s^{\prime} \leq 0$.
Since $\vp : [\tau, 0] \to X$ is continuous, there is $\delta_{1} > 0$ such that
\[
\| \vp(s) - \vp(s^{\prime}) \| \leq \ve,
\]
for $| s - s^{\prime}| \leq \delta_{1}$. In addition, since the functions $\ta \mapsto \rho(\ta) \|\vp(\ta)\|^{2}$, for $s \in [\tau, 0]$,
are integrable on $(- \infty, 0]$, an standard argument approximating by continuous functions and applying the Lebesgue dominated convergence
theorem shows that there exists $\delta_{2} > 0$ such that
\[
\left(\int_{-\infty}^{0} \rho(\ta) \|\vp(s + \ta) - \vp(s^{\prime} + \ta)\|^{2} d \ta \right)^{1/2} \leq \ve,
\]
for $| s - s^{\prime}| \leq \delta_{2}$. Hence, for $| s - s^{\prime}| \leq \delta = \min \{\delta_{1}, \delta_{2} \}$, we have
\[
\| \vp_{s} - \vp_{s^{\prime}} \|_{\bb} = \| \vp(s) - \vp(s^{\prime}) \| +  \left(\int_{-\infty}^{0} \rho(\ta) \|\vp(s + \ta) - \vp(s^{\prime} + \ta)\|^{2} d \ta \right)^{1/2} \leq 2\ve,
\]
This completes the proof of the second assertion.

\end{proof}

\setcounter{equation}{0}
\setcounter{remark}{0}
\setcounter{pr}{0}
\setcounter{de}{0}
\setcounter{theorem}{0}
\setcounter{lema}{0}
\setcounter{co}{0}
\setcounter{example}{0}

\section{Linear systems } \label{linear}
Motivated by the model discussed in the previous section, we study the class of abstract linear systems of the form \eqref{equ4.1}--\eqref{equ4.2}. Specifically, we assume that $\Lambda : \mathcal{B} \to X$ is a bounded linear operator, $h : [0, a] \to X$ is integrable, and the initial condition $\varphi$ belongs to $\mathcal{B}_{\tau}$.

\begin{de} \label{D8} A function $x\colon (-\infty, a] \to X$ is said to be a mild
solution of problem \eqref{equ4.1}-\eqref{equ4.2}
if $x$ is continuous on $[s, a]$, the condition \eqref{equ4.2} holds for $\vp \in \bb_{\tau}$, and the integral equation
\begin{equation}
x(t) = T(t -s) \vp(0) + \int_{s}^{t} T(t -\xi) \La(x_{r(\xi)}) d \xi + \int_{s}^{t} T(t -\xi) h(\xi) d \xi, \qquad s \leq t \leq a,
\label{equ3.3}
\end{equation}
is satisfied. 
\end{de}
\begin{remark}\rm 
Similarly as in Theorem~3.1, but considering certain modifications in the proof, we can affirm that  problem (\ref{equ4.1})-(\ref{equ4.2}) has a unique mild
solution defined on $(-\infty, a]$ for all $0 \leq s < a$. In fact, let $0 \leq s < a$, and consider the closed convex subset $\hat{Y}=\{u\in C([s,a],X): u(s)=\vp(0)\}$, where $\vp\in\mathcal{B}_{\tau}$. In this case, for $u\in\hat{Y}$ we identify $u\in C([s,a],X)$ with its extention to $(-\infty,s]$ which is defined by $u(\theta)=\vp(\theta-s)$ for all $\theta\in(-\infty,s]$.  We consider the map $\Gamma$ on $\hat{Y}$ defined by 
\[
\Gamma(u)(t)=T(t -s) \vp(0) + \int_{s}^{t} T(t -\xi) \La(u_{r(\xi)}) d \xi + \int_{s}^{t} T(t -\xi) h(\xi) d \xi, \quad s \leq t \leq a.
\]
For every $u,v\in\hat{Y}$ we can estimate
\begin{eqnarray*}
	\|\Gamma (u) (t) - \Gamma (v) (t) \|  & \leq & N \int_{s}^{t} \|\La\|\|  u_{r(\xi)} -
	v_{r(\xi)} \|_{\bb} d \xi \\
	& \leq &  N \|\La\| K_{a} \int_{s}^{t}  \max_{s \leq \theta \leq r(\xi)} \| u(\theta) - v(\theta) \| d \xi \\
	& \leq &  N K_{a}\|\La\|(t-s)\|u-v\|_{\infty},
\end{eqnarray*}
for all $0 \leq t \leq a$. Inductively, we can show that for every $n\in\mathbb{N}$: 
\[
\|\Gamma^{n} (u) (t) - \Gamma^{n} (v) (t) \|\leq \frac{(N K_{a} \|\La\|(t-s))^{n}}{n!} \| u - v \|_{\infty}, \qquad t\in[0,a].
\]
Therefore, for $n$ large enough, the map $\Gamma^{n}$ is a contraction. Consequently, the map $\Gamma$ has a unique fixed point $x$, which is the unique mild solution of problem (\ref{equ4.1})-(\ref{equ4.2}).

\end{remark}
\subsection{The homogeneous case }
Throughout this subsection, we assume that $h(t) = 0$ for all $t \geq 0$. Let $ x(\cdot, s, \vp)$ be the  mild solution of problem  (\ref{equ4.1})-(\ref{equ4.2}).
We define the map
\begin{equation}
U(t, s) \vp = x_{t}(\cdot, s, \vp), \;\; t \geq s. \label{equ4.3}
\end{equation}

We recall the concept of evolution family. We denote $\Delta = \{(t,s)\in \mathbb{R}^2: 0 \leq s \leq t, \; t \leq a \}$.

\begin{de} \label{D3} A family $(U(t,s))_{(t,s)\in\Delta} $ is said to be an evolution family on a Banach space $Y$
if the following conditions are fulfilled:
\begin{itemize}
\item[(i)] The function $U(t,s) : Y \to Y$ is a bounded linear map for all $s \leq t$.
\item[(ii)] $U(s,s) = I$, for $0 \leq s \leq a$.
\item[(iii)] $U(t,s_{1}) = U(t,s_{2}) U(s_{2},s_{1})$ for $0 \leq s_{1} \leq s_{2} \leq t \leq a$.
\item[(iv)] The function $U(\cdot, \cdot) : \Delta   \to \lc(Y)$ is strongly continuous.
\end{itemize}
\end{de}

\begin{pr} \label{P4} The family $(U(t,s))_{t \geq s} $ given by
\eqref{equ4.3}  is an evolution family  on $\bb_{\tau}$.
\end{pr}
\begin{proof} By (A)-(i), it follows that $U(t,s) \vp \in \bb_{\tau}$ for all $\vp \in \bb_{\tau}$.
Moreover, as usual the uniqueness of solutions implies
that $U(t, s)$ is a linear mapping. In addition, it follows from \eqref{equ3.3} that
\begin{equation}
\|x(t)\| \leq N \|\vp(0)\| + N \|\La\| \int_{s}^{t} \|x_{r(\xi)}\|_{\bb} d \xi, \label{equ3.5}
\end{equation}
where we abbreviate $x(t) = x(t,s, \vp)$.
To estimate $\|x_{r(\xi)}\|_{\bb}$ we consider two cases.

(i) If $s \leq r(\xi)$, then
\[
\|x_{r(\xi)}\|_{\bb} \leq K(r(\xi) -s) \max_{s \leq \sa \leq r(\xi)} \| x(\sa)\|  + M(r(\xi) -s) \|\vp\|_{\bb}.
\]

(ii) If $ r(\xi) < s$, then
\[
\|x_{r(\xi)}\|_{\bb} = \|\vp_{r(\xi) -s}\|_{\bb} \leq \|\vp\|_{\tau}.
\]
Therefore, in any case, we have
\[
\|x_{r(\xi)}\|_{\bb} \leq K_{a} \max_{s \leq \sa \leq \xi} \| x(\sa)\|  + M_{a} \|\vp\|_{\tau}.
\]
Substituting in \eqref{equ3.5}, and applying the Gronwall-Bellman lemma, we obtain
\[
\|x(t)\| \leq C_{1} e^{C_{2} (t -s)} \|\vp\|_{\tau}
\]
for some generic constants $C_{1}, C_{2} \geq 0$ independent of $\vp$. Hence,
\[
\|x_{t}\|_{\bb} \leq K_{a} \max_{s \leq \sa \leq t} \| x(\sa)\|  + M_{a} \|\vp\|_{\tau}
\]
which can be abbreviated as
\[
\|x_{t}\|_{\bb} \leq  C_{1} e^{C_{2} (t -s)} \|\vp\|_{\tau},
\]
and
\[
\|x_{t}\|_{\tau} \leq  C_{1} e^{C_{2} (t -s)} \|\vp\|_{\tau},
\]
where we assume that $C_{1} \geq 1$. This shows that $U(\cdot,\cdot)$ satisfies (i) of Definition~\ref{D3}.
The assertion (ii) of Definition~\ref{D3} is immediate.

In order to prove (iii) of Definition~\ref{D3}, assume that  $0 \leq s_{1} \leq s_{2} \leq s_{3}$.
Let $x = x(\cdot, s_{1}, \vp) $ be the mild solution of \eqref{equ4.1} on $[s_{1}, a]$ with initial condition
$x_{s_{1}}(\cdot, s_{1}, \vp) = \vp$. We set $\psi = x_{s_{2}}(\cdot, s_{1}, \vp)$, and let $y = y(\cdot, s_{2}, \psi)$
be the mild solution of \eqref{equ4.1} on $[s_{2}, a]$ with initial condition $y_{s_{2}}(\cdot, s_{2}, \psi) = \psi$.
We will show that
\[
x(t, s_{1}, \vp) = y(t, s_{2}, \psi),
\]
for all  $ s_{2} \leq t \leq s_{3}$. It follows from \eqref{equ3.3} that
\begin{eqnarray*}
x(t) & = & T(t -s_{1}) \vp(0) + \int_{s_{1}}^{t} T(t - s) \La(x_{r(s)}) d s \\
& = & T(t -s_{1}) \vp(0) + \int_{s_{1}}^{s_{2}} T(t - s) \La(x_{r(s)}) d s + \int_{s_{2}}^{t} T(t - s) \La(x_{r(s)}) d s,
\end{eqnarray*}
and
\begin{eqnarray*}
y(t) & = & T(t -s_{2}) \psi(0) + \int_{s_{2}}^{t} T(t - s) \La(y_{r(s)}) d s \\
& = & T(t -s_{2}) \Big[ T(s_{2} -s_{1}) \vp(0) + \int_{s_{1}}^{s_{2}} T(s_{2} - s) \La(x_{r(s)}) d s \Big] +
\int_{s_{2}}^{t} T(t - s) \La(y_{r(s)}) d s.
\end{eqnarray*}
Therefore,
\[
y(t) - x(t) = \int_{s_{2}}^{t} T(t - s) \La(y_{r(s)} - x_{r(s)}) d s,
\]
which implies that
\[
\|y(t) - x(t)\| \leq  N \|\La\| \int_{s_{2}}^{t} \|y_{r(s)} - x_{r(s)}\|_{\bb} d s.
\]
We next estimate $\|y_{r(s)} - x_{r(s)}\|_{\bb}$. We analyze two cases:

a) If $s_{2} \leq r(s) $, then
\[
 \|y_{r(s)} - x_{r(s)}\|_{\bb} \leq K_{a} \max_{s_{2} \leq \xi \leq r(s)} \|y(\xi) - x(\xi) \|.
\]

b) If $ r(s) < s_{2}$, then $y_{r(s)} = x_{r(s)}$.

Combining cases a) and b), we can affirm that
\begin{eqnarray*}
\|y(t) - x(t)\| & \leq &  N \|\La\| K_{a} \int_{s_{2}}^{t}  \max_{s_{2} \leq \xi \leq r(s)} \|y(\xi) - x(\xi) \| d s \\
& \leq & N \|\La\| K_{a} \int_{s_{2}}^{t}  \max_{s_{2} \leq \xi \leq t} \|y(\xi) - x(\xi) \| d s,
\end{eqnarray*}
and applying the Gronwall-Bellman lemma we obtain that $x(t) = y(t)$.

Finally, we study the strong continuity of $U(\cdot,\cdot)$. Initially, we study
$U(\cdot, s) \vp $ for fixed $0 \leq s <a$. Let $x = x(\cdot, s, \vp) $ be the mild solution of problem \eqref{equ4.1}-\eqref{equ4.2}.
Since $x(\cdot)$ is continuous, it follows from Lemma~\ref{L3.1} that the function $U(\cdot, s) \vp $ is $\bb_{\tau}$-continuous.

We now study the continuity of $U(t, \cdot) \vp $ on $[0, t]$ for fixed $0 <t  \leq a$. Assuming $\delta \geq 0$ and $s + \delta \leq t$,
and using  that
\[
U(t, s) \vp = U(t, s + \delta)  U(s+ \delta, s) \vp
\]
which implies that
\[
U(t, s) \vp - U(t, s + \delta) \vp = U(t, s + \delta) [ U(s+ \delta, s) \vp - \vp].
\]
By the previous part of this proof, we know that  $\|U(s+ \delta, s) \vp - \vp\|_{\tau} \to 0$ as $\delta \to 0$. Moreover, since
$U(\cdot,\cdot)$ is uniformly bounded on $\Delta$, we conclude that the function $U(t, \cdot) \vp$ is continuous for the norm
in $\bb_{\tau}$ and that this property is independent of $t \in [0, a]$. From this fact  we conclude that the function
$U(\cdot,\cdot)$ is continuous on $\Delta$ for the norm in $\bb_{\tau}$.

\end{proof}

In the unperturbed case where $\La = 0$, $h=0$ and $s = 0$, the mild solution of the problem \eqref{equ4.1}--\eqref{equ4.2} is governed by a strongly continuous semigroup of linear operators $(V(t))_{t \geq 0}$ on $\mathcal{B}$, which corresponds to the translation semigroup associated with $(T(t))_{t \geq 0}$. Specifically, for any $\vp \in \mathcal{B}$ and $\ta \leq 0$, this semigroup is given by
\[
[V(t) \vp](\ta) = 
\begin{cases} 
T(t + \ta) \vp(0), &  t + \ta \geq 0, \\
\vp(t + \ta), &  t + \ta \leq 0.
\end{cases}
\]

\begin{remark}
 \rm   
In what follows,
\begin{itemize}
\item we assume that $\bb$ satisfies axiom \textup{\textbf{(C-2)}}.

\medskip

\item for $x \in X$ and $\la \in \mathbb{R}$, we denote by $F_{\la} x$ the function
\[
[F_{\la} x](\ta)  =  e^{\la \ta} x, \; - \infty < \ta \leq 0.
\]
\end{itemize}
\end{remark}

The relationship among  $C_{0}((-\infty, 0], X)$,  $\bb_{\tau}$ and $V(t)$ is established in the following result.
\begin{lema} \label{L4.1} The family  $(V(t))_{t \geq 0}$ is a strongly continuous semigroup of linear operators on $\bb_{\tau}$. Moreover,
$C_{0}((-\infty, 0], X)$ is a closed subspace of $\bb_{\tau}$ invariant under $(V(t))_{t \geq 0}$, and
$\|\vp\|_{\tau} \leq Q \|\vp\|_{\infty}$ for all $\vp \in C_{0}((-\infty, 0], X)$.
\end{lema}
\begin{proof} Most of the assertions in this statement are immediate, so we will omit their proofs.
We will  only justify that  the map  $V$ is strongly continuous on $\bb_{\tau}$. However, this assertion is also an immediate consequence of
Lemma~\ref{L3.1}. In fact, for $\vp \in \bb_{\tau}$, we define
\[
x(t) = \left\{ \begin{array}{ccl} T(t) \vp(0), & & t \geq 0, \\ \vp(t), & & t \leq 0, \end{array} \right.
\]
then $x_{t} = V(t) \vp$ for $t \geq 0$. Applying Lemma~\ref{L3.1} we conclude that $V(\cdot) \vp$ is a continuous function.

\textcolor{red}{}
\end{proof}

\begin{remark}
    \rm
In what follows, we denote by $N_{2}$ a constant bounding the translation semigroup, so that
\[
\|V(t)\| \leq N_{2} \quad \text{for all } 0 \leq t \leq a.
\]

\end{remark}

Before addressing the main result of this section, we construct the necessary analytical machinery by establishing some technical lemmas on the integration of regulated functions.

\begin{lema} \label{L4.2} Let $f : [a, b] \to X $ be a regulated function.
Then $f$ is Riemann integrable.
\end{lema}
\begin{proof}
Assume initially that $f$ is a step function. In similar way as for real functions, it follows from \cite{Gordon}
that $f$ is Riemann integrable. For a general regulated
function $f$ we use that $f$ is the uniform limit of a sequence of step functions.
\end{proof}

\begin{lema} \label{L4.3} For $t\ge s$, let $q^{0} : [s, t] \to C_{0}((-\infty, 0], X) $ be a regulated function.
Then the function $q : [s, t] \to C_{0}((-\infty, 0], X) $ given by
\[
q(\xi) = V(t -\xi) q^{0}(\xi), \; s \leq \xi \leq t,
\]
is also regulated, and
\[
\left[\int_{s}^{t} V(t -\xi) q^{0}(\xi) d \xi \right] (\ta) = \int_{s}^{t} [V(t -\xi) q^{0}(\xi)](\ta) d \xi
\]
for all $- \infty < \ta \leq 0$.
\end{lema}

\begin{proof}
Let $\xi \in [s, t)$ and let $(\xi_{n})_{n}$ be a nonincreasing sequence that converges to $\xi$.
Since the function $q^{0}$ is regulated, we can affirm that $q^{0}(\xi_{n}) \to q^{0}(\xi^{+})$ as $n \to \infty$. Consequently,
\begin{eqnarray*}
V(t - \xi_{n}) q^{0}(\xi_{n}) - V(t - \xi) q^{0}(\xi^{+}) & = &
V(t - \xi_{n}) [q^{0}(\xi_{n}) - q^{0}(\xi^{+})] \\
& & + [V(t - \xi_{n}) q^{0}(\xi^{+}) - V(t - \xi) q^{0}(\xi^{+})]
\end{eqnarray*}
It is clear that the first term on the right hand side converges to zero. Using  that $(V(t))_{t \geq 0}$ is a
strongly continuous semigroup, we infer that the second term on the right hand side also converges to zero.
Combining these assertions, we obtain
that $q$ has right-sided limits. Proceeding in similar way, we can show that $q$ has left-sided limits.

\noindent Consequently,  there exists ${\displaystyle \int_{s}^{t} V(t -\xi) q^{0}(\xi) d \xi}$ as a Riemann integral in $C_{0}((-\infty, 0], X) $.
In addition, for $\ta \leq 0$, $E_{\ta} :  C_{0}((-\infty, 0], X) \to X$ given by $E_{\ta}(\vp) = \vp(\ta)$ is a bounded linear map. This implies that
\begin{eqnarray*}
E_{\ta} \left(\int_{s}^{t} V(t -\xi) q^{0}(\xi) d \xi \right) & = & \left[\int_{s}^{t} V(t -\xi) q^{0}(\xi) d \xi \right] (\ta) \\
& = & \int_{s}^{t} E_{\ta} \left( V(t -\xi) q^{0}(\xi) \right) d \xi \\
& = & \int_{s}^{t}  \left[ V(t -\xi) q^{0}(\xi) \right] (\ta) d \xi,
\end{eqnarray*}
which completes the proof.
\end{proof}

Inspired by the approach in \cite{BMM}, we are now in a position to develop our framework. Since $(T(t))_{t \geq 0}$ is a strongly continuous semigroup, we know that (\cite{EN,Pa}) there exist $\widetilde{M}  \geq 1$
and $\om \in \real$ such that
\[
\|T(t)\| \leq \widetilde{M} e^{\om t}, \; t \geq 0.
\]
In this case, $\|R(\la, A) \| \leq \frac{\widetilde{M}}{\la - \om}$ for all $ \la > \om$.
Hence, we can take $N_{1}$ and $\om_{1} > \max \{0, \om\} $ large enough
so that
\begin{equation}
\|R(\la, A)\| \leq \frac{N_{1}}{\la} \label{equ4.8}
\end{equation}
for all $\la \geq \om_{1}$.
We also observe  that $\la R(\la, A) x \longrightarrow x$ as $\la \to \infty$ for all $x \in X$.
In what follows, we assume that $\la \geq \om_{1}$, which implies that
$F_{\la} : X \to C_{0}((-\infty, 0], X)$ is well defined.

\begin{lema} \label{L4.4} Let $g \in G([0, a], X)$.
Then
\[
\lim_{\la \to \infty} \int_{s}^{t} V(t -\xi) \la F_{\la} R(\la, A) g(\xi) d \xi
\]
exists in $C_{0}((-\infty, 0], X)$, and the convergence is uniform for $0 \leq s \leq t \leq a$.
\end{lema}

\begin{proof} Since $F_{\la}$ is a bounded linear map, the function $ F_{\la} R(\la, A) g$ is a regulated function from
$[0, a]$ into $C_{0}((-\infty, 0], X) $. It follows from Lemma~\ref{L4.3} that the integral in
\[
W_{\la}(t, s) = \int_{s}^{t} V(t -\xi) \la F_{\la} R(\la, A) g(\xi) d \xi.
\]
exists as a Riemann integral in  $C_{0}((-\infty, 0], X)$. In addition,
\[
[W_{\la}(t, s)](\ta) = \int_{s}^{t} [V(t -\xi) \la F_{\la} R(\la, A) g(\xi)](\ta) d \xi
\]
for all $- \infty < \ta \leq 0$.

For $t + \ta \geq s$, using definitions of $(V(t))_{t \geq 0}$ and $F_{\la}$, we have that
\begin{eqnarray*}
[W_{\la}(t, s)](\ta) & = & \int_{s}^{t + \ta} [V(t -\xi) \la F_{\la} R(\la, A) g(\xi)](\ta) d \xi \\
& & + \int_{t + \ta}^{t} [V(t -\xi) \la F_{\la} R(\la, A) g(\xi)](\ta) d \xi \\
& = & \int_{s}^{t + \ta} T(t -\xi + \ta) \la  R(\la, A) g(\xi) d \xi + \int_{t + \ta}^{t}  \la e^{\la(t -\xi + \ta)} R(\la, A) g(\xi) d \xi.
\end{eqnarray*}
Hence, for $\la$ and $\mu$ large enough, we obtain
\begin{align*}
\lefteqn{\|[W_{\la}(t, s)](\ta) - [W_{\mu}(t, s)](\ta)\|}\\
 \leq & N \int_{s}^{t + \ta} \| [\la  R(\la, A) - \mu  R(\mu, A)] g(\xi)\| d \xi \\
& + \int_{t + \ta}^{t}  \|\la e^{\la(t -\xi + \ta)} R(\la, A) - \mu e^{\mu(t -\xi + \ta)} R(\mu, A)\| \|g(\xi)\| d \xi \\
\leq & N \int_{0}^{a} \| [\la  R(\la, A) - \mu  R(\mu, A)] g(\xi)\| d \xi +
\int_{ \ta}^{0}  (\la e^{\la \xi} \|R(\la, A)\| + \mu e^{\mu \xi} \|R(\mu, A)\|  d \xi ) \|g\|_{\infty} \\
\leq & N \int_{0}^{a} \| [\la  R(\la, A) - \mu  R(\mu, A)] g(\xi)\| d \xi +
N_{1} \left(\frac{1}{\la} + \frac{1}{\mu}\right)  \|g\|_{\infty}.
\end{align*}
The first term on the right hand side converges to zero as $\la, \mu \to \infty$ by our previous remark and applying the Lebesgue
dominated convergence theorem. Likewise, the second term on the right hand side converges to zero as $\la, \mu \to \infty$.

For $t + \ta < s$ and every $\xi \in [s, t]$ we have that $t -\xi + \ta \leq 0$. Therefore, using again
the  definitions of $(V(t))_{t \geq 0}$ and $F_{\la}$, we obtain that
\[
\|[W_{\la}(t, s)](\ta) - [W_{\mu}(t, s)](\ta)\|  \leq N_{1} \left(\frac{1}{\la} + \frac{1}{\mu}\right)  \|g\|_{\infty}.
\]
Moreover, we note that in both cases the convergence is independent of $\ta \leq 0$.
Hence, combining both cases, we conclude that the limit
\[
\lim_{\la \to \infty} \int_{s}^{t} V(t -\xi) \la F_{\la} R(\la, A) g(\xi) d \xi
\]
exists in $C_{0}((-\infty, 0], X)$. Furthermore, by the preceding estimates the convergence is uniform for
$0 \leq s \leq t \leq a$.

\end{proof}

\begin{remark}
    \rm 
It follows from Lemma~\ref{L4.4} that the function $q : \Delta \to  C_{0}((-\infty, 0], X)$ given by
\[
q(t,s)= \lim_{\la \to \infty} \int_{s}^{t} V(t -\xi) \la F_{\la} R(\la, A) g(\xi) d \xi
\]
is continuous.
\end{remark}

We now define inductively a sequence of operators on $\mathcal{B}_{\tau}$ by
\begin{eqnarray*}
U_{0}(t, s) \vp & = & V(t - s) \vp, \\
U_{n}(t, s) \vp & = & \lim_{\la \to \infty} \int_{s}^{t} V(t -\xi) \la F_{\la} R(\la, A) \La(U_{n -1}(r(\xi), s) \vp) \, d \xi, \quad n \geq 1,
\end{eqnarray*}
for all $0 \leq s \leq t \leq a$ and $\vp \in \mathcal{B}_{\tau}$. Here, for any $\xi \in [s,t]$ such that $r(\xi) - s < 0$, the terms under the integral sign are defined as
\begin{equation}
U_{n-1}(r(\xi), s) \vp := 
\begin{cases} 
\vp_{r(\xi) - s}, & \text{if } n = 1, \\
0, & \text{if } n \geq 2.
\end{cases}
\end{equation}
We define $U_{\La}(t, s)$ by the Dyson--Phillips type expansion
\begin{equation}
U_{\La}(t, s) = \sum_{n = 0}^{\infty} U_{n}(t, s), \quad 0 \leq s \leq t \leq a. \label{equ4.5}
\end{equation}

\begin{pr}\label{serie D-P}
The series \eqref{equ4.5} converges in $\lc(\bb_{\tau})$ uniformly on $\Delta$.    
\end{pr}
\begin{proof}
    Since $(V(t))_{t \geq 0}$ is a strongly continuous semigroup, we can assume that
the function $g$ given by $ g(\xi) = \La(U_{n -1}(r(\xi), s) \vp)$ is a regulated function. Using now Lemma~\ref{L4.4}, for every $n \geq 1$
we infer that $U_{n}$ is well defined.

First, we establish the following preliminary estimations:
\[
\|U_{0}(t, s) \vp\|_{\mathcal{\tau}}=\|V(t-s) \vp\|_{\tau} \leq N_2 \|\vp\|_{\tau}, \quad t\geq s \geq 0,
\]
and 
\[
\|U_{0}(r(\xi),s)\vp\|_{\mathcal{B}} =\|\vp_{r(\xi)-s}\|_{\mathcal{B}}\leq \|\vp\|_{\tau},
\]
for all $\xi \in [s,t]$ such that $r(\xi)-s<0$. The last inequality follows since $r(\xi)-s \in [\tau,0]$.

On the other hand, for every $0\leq s \leq t \leq a$, we infer 

\begin{equation*}
\begin{aligned}
    \lefteqn{\left\|\int_{s}^t V(t-\xi)\lambda F_{\lambda}R(\lambda,A)\Lambda(U_0(r(\xi),s)\vp)\,d\xi \right\|_\tau} \\
    &\leq \int_{s}^t \|V(t-\xi)\lambda F_{\lambda}R(\lambda,A)\Lambda(U_0(r(\xi),s)\vp)\|_\tau \,d\xi \\
    &\leq \lambda N_2 \int_{s}^t \|F_\lambda R(\lambda,A)\Lambda(U_0(r(\xi),s)\vp) \|_{\tau}\,d\xi.
\end{aligned}
\end{equation*}

Observe that the function $\theta \mapsto F_\lambda R(\lambda,A)\Lambda(U_0(r(\xi),s)\vp)(\theta)$ belongs to $C_b((-\infty,0],X)$. Hence, from Remark~\ref{R2.3}, there exists $Q>0$ such that
\begin{equation*}
\begin{split}
 \|F_\lambda R(\lambda,A)\Lambda(U_0(r(\xi),s)\vp)\|_{\tau}&\leq Q\|F_\lambda R(\lambda,A)\Lambda(U_0(r(\xi),s)\vp)\|_{\infty}\\
&\leq Q\|R(\lambda,A)\Lambda(U_0(r(\xi),s)\vp)\|_X.
\end{split}
\end{equation*}
Moreover, from the estimation of the resolvent operator and the fact that $\Lambda$ is a bounded linear map, we obtain
\[\|F_\lambda R(\lambda,A)\Lambda(U_0(r(\xi),s)\vp)\|_{\tau} \leq Q\frac{N_1}{\lambda}\|\Lambda(U_0(r(\xi),s)\vp)\|_X\leq Q\frac{N_1}{\lambda}\|\Lambda\|\|U_0(r(\xi),s)\vp\|_{\mathcal{B}}.\]
Consequently,
\begin{eqnarray*}
    \left\| \int_{s}^tV(t-\xi)\lambda F_{\lambda}R(\lambda,A)\Lambda(U_0(r(\xi),s)\vp)d\xi\right\|_\tau&\leq QN_1N_2\|\Lambda\|\displaystyle\int_{s}^{t}\|U_0(r(\xi),s)\vp\|_{\mathcal{B}}d\xi.
\end{eqnarray*}
Now, since
\begin{eqnarray*}
    \int_{s}^{t}\|U_0(r(\xi),s)\vp\|_{\mathcal{B}}d\xi &=& \int_{\{\xi \in [s,t] \,:\, r(\xi)\geq s\}}\|U_0(r(\xi),s)\vp\|_{\mathcal{B}}d\xi \\
     & &+\int_{\{\xi \in [s,t]\,:\,r(\xi)< s\}}\|U_0(r(\xi),s)\vp\|_{\mathcal{B}}d\xi \\
    &\leq&  \int_{\{\xi \in [s,t] \,:\, r(\xi)\geq s\}}\|U_0(r(\xi),s)\vp\|_{\tau}d\xi\\
    & & + \int_{\{\xi \in [s,t]\,:\,r(\xi)< s\}}\|\varphi_{r(\xi)-s}\|_{\mathcal{B}}d\xi\\
    &\leq& \int_{s}^t N_2 \|\varphi\|_\tau d\xi+\int_{s}^t\|\varphi\|_{\tau}d\xi\\
    &=&(N_2+1)(t-s)\|\varphi\|_{\tau},
\end{eqnarray*}
we deduce that
\[ \left\| \int_{s}^tV(t-\xi)\lambda F_{\lambda}R(\lambda,A)\Lambda(U_0(r(\xi),s)\vp)d\xi\right\|_\tau \leq QN_1N_2\|\Lambda\|(N_2+1)(t-s)\|\vp\|_{\tau},\]
and taking $\lambda \to \infty$, we obtain
\[\|U_1(t,s)\|_{\tau} \leq Q_1N_1N_2\|\Lambda\|(N_2+1)(t-s)\|\vp\|_{\tau}, \quad \text{for all}\quad t \geq s. \]
Similarly,
\begin{eqnarray*}
\|U_{2}(t, s) \vp\|_{\tau} & \leq & N_1N_2 Q\|\Lambda\| \int_{s}^{t} \|U_{1}(r(\xi), s) \vp) \|_{\tau}d \xi \\
&=&N_1N_2Q\|\Lambda\| \int_{\{\xi \in [s,t] \,:\, r(\xi)\geq s\}} \|U_{1}(r(\xi), s) \vp) \|_{\tau}d \xi\\
&\leq & N_1N_2Q\|\Lambda\|QN_1N_2\|\Lambda\|(N_2+1)\|\vp\|_{\tau}\int_{s}^{t}(r(\xi)-s)d\xi\\
&\leq & N_1^2N_2^2Q^2\|\Lambda\|^2(N_2+1)\|\vp\|_{\tau}\int_{s}^{t}(\xi-s)d\xi\\
&=& N_1^2N_2^2Q^2\|\Lambda\|^2(N_2+1)\|\vp\|_{\tau}\frac{(t-s)^2}{2}, \quad \text{for all}\quad t \geq s.
\end{eqnarray*}
Proceeding inductively, we can show that for every $n\geq 1$,
\begin{equation}
\|U_{n}(t, s) \vp\|_{\tau}  \leq (N_1N_{2}Q \|\La\|)^{n}(N_2+1)\frac{(t -s)^{n}}{n!}\|\vp\|_{\tau} \label{equ4.9}, \quad \text{for all}\quad t \geq s.
\end{equation}
This shows that the series $\sum_{n = 0}^{\infty} U_{n}(t, s)$ converges and the
convergence is uniform in $\Delta$. Furthermore, the convergence holds in $\mathcal{L}(\mathcal{B}_\tau)$, since
\[\|U_\Lambda(t,s)\vp\|_{\tau} \leq (N_2+1)e^{N_1N_2Q\|\Lambda\|(t-s)}\|\vp\|_{\tau}.\]

\end{proof}

In the next, we present the main result of this subsection

\begin{theorem}\label{P4.1}
The following statements hold:

\begin{itemize}
\item[(i)] The family $(U_{\La}(t, s))_{t\geq s}$ satisfies the translation property
\begin{equation} \label{equ4.12}
[U_{\La}(t, s) \vp](\ta) = \left\{ \begin{array}{rcl} [U_{\La}(t + \ta, s) \vp] (0), & & t + \ta \geq s, \\
\vp(t + \ta -s), & & t + \ta \leq s, \end{array} \right.
\end{equation}
for all $\ta \leq 0$.
\item[(ii)] For every $\vp \in \bb_{\tau}$ and $0 \leq s \leq t$, the relation 
\begin{equation}
U_{\La}(t, s) \vp = U_{0}(t, s) \vp + \lim_{\la \to \infty} \int_{s}^{t} V(t -\xi) \la F_{\la} R(\la, A) \La
\left(U_{\La}(r(\xi), s)  \vp \right) d \xi \label{equ4.4}
\end{equation}
holds, where the operator $U_{\Lambda}(r(\xi),s)\varphi$ under the integral sign is extended to the past $r(\xi)<s$ by setting
\[
U_{\Lambda}(r(\xi),s)\varphi=\varphi_{r(\xi)-s}.
\]

\item[(iii)] The mild solution of problem \eqref{equ4.1}-\eqref{equ4.2} is given by
\[
x(t, s, \vp) = \left\{\begin{array}{rcl} [U_{\La}(t, s) \vp] (0), & & t \geq s, \\
\vp(t -s), & & t \leq s, \end{array} \right.
\]
and its state history satisfies
\[
x_{t}(\cdot, s, \vp) = U_{\La}(t,s) \vp.
\]
\end{itemize}
\end{theorem}
\begin{proof} 
\noindent{\bf(i)} To establish this property, we will show that $U_{n}(t, s)$ satisfies the same property.
It follows from the definition of $(V(t))_{t \geq 0}$ that the assertion is immediate for $n = 0$.

Assume that $t + \ta -s \geq 0$ and $n\geq 1$. In this case,
\begin{eqnarray*}
[U_{n}(t, s) \vp](\ta) & = & \left[\lim_{\la \to \infty} \int_{s}^{t} V(t -\xi) \la F_{\la} R(\la, A) \La(U_{n -1}(r(\xi), s) \vp) d \xi \right](\ta) \\
& = & \lim_{\la \to \infty} \int_{s}^{t} \left[V(t -\xi) \la F_{\la} R(\la, A) \La(U_{n -1}(r(\xi), s) \vp)\right](\ta) d \xi  \\
& = & \lim_{\la \to \infty} \int_{s}^{t + \ta} T(t + \ta - \xi) \la  R(\la, A) \La(U_{n -1}(r(\xi), s) \vp) d \xi \\
& & + \lim_{\la \to \infty} \int_{t + \ta}^{t} \la e^{\la(t -\xi + \ta)}  R(\la, A) \La(U_{n -1}(r(\xi), s) \vp) d \xi.
\end{eqnarray*}
Using Lebesgue's dominated convergence theorem, and our previous estimates \eqref{equ4.8} and \eqref{equ4.9}, we infer that
\[
\lim_{\la \to \infty} \int_{t + \ta}^{t} \la e^{\la(t -\xi + \ta)}  R(\la, A) \La(U_{n -1}(r(\xi), s) \vp) d \xi = 0,
\]
which implies that
\[
[U_{n}(t, s) \vp](\ta) =  \lim_{\la \to \infty} \int_{s}^{t + \ta} T(t + \ta - \xi) \la  R(\la, A) \La(U_{n -1}(r(\xi), s) \vp) d \xi.
\]
On the other hand,
\begin{eqnarray*}
[U_{n}(t + \ta, s) \vp](0) & = & \left[\lim_{\la \to \infty} \int_{s}^{t + \ta} V(t + \ta -\xi) \la F_{\la} R(\la, A) \La(U_{n -1}(r(\xi), s) \vp) d \xi \right](0) \\
& = & \lim_{\la \to \infty} \int_{s}^{t + \ta} T(t + \ta -\xi) \la R(\la, A) \La(U_{n -1}(r(\xi), s) \vp) d \xi,
\end{eqnarray*}
which shows that
\[
[U_{n}(t, s) \vp](\ta) = [U_{n}(t + \ta, s) \vp](0).
\]
For $t + \ta -s < 0$, and proceeding as above, we obtain that
\[
[U_{n}(t, s) \vp](\ta)= 0.
\]
Substituting in the series \eqref{equ4.5}, we obtain that
$U_{\La}(t, s)$ satisfies the translation property \eqref{equ4.12}.
This establishes the assertion  (i).

\medskip

\noindent{\bf(ii)} To establish (ii), let $t\geq s$. Note that
\begin{eqnarray*}
U_{\La}(t, s) \vp & = & U_{0}(t, s) \vp+ \lim_{k\to\infty}\sum_{n = 1}^{k} U_{n}(t, s) \vp,
\end{eqnarray*}
therefore, it suffices to show that
\begin{eqnarray*}
   \lim_{k\to\infty}\sum_{n = 1}^{k} U_{n}(t, s) \vp= \lim_{\la \to \infty} \int_{s}^{t} V(t -\xi) \la F_{\la} R(\la, A) \La
\left(U_{\La}(r(\xi), s)  \vp \right) d \xi.
\end{eqnarray*}
Indeed,  by Lemma~\ref{L4.4} and using Lebesgue's dominated convergence theorem, it follows that 
\begin{eqnarray*}
    \lim_{k\to\infty}\sum_{n = 1}^{k} U_{n}(t, s) \vp&=&\lim_{k\to\infty}\sum_{n = 1}^{k} \lim_{\la \to \infty} \int_{s}^{t} V(t -\xi) \la F_{\la} R(\la, A) \La(U_{n -1}(r(\xi), s) \vp) d \xi\\
    &=&\lim_{\lambda \to \infty}\int_{s}^t V(t -\xi) \la F_{\la} R(\la, A) \La\left(\sum_{n=1}^{\infty}U_{n -1}(r(\xi), s) \vp\right) d \xi\\
    &=&\lim_{\lambda \to \infty}\int_{\{\xi \in [s,t]\,:\, r(\xi) \geq s\}} V(t -\xi) \la F_{\la} R(\la, A) \La\left(\sum_{n=1}^{\infty}U_{n -1}(r(\xi), s) \vp\right) d \xi\\
    & & +\lim_{\lambda \to \infty}\int_{\{\xi \in [s,t]\,:\, r(\xi) < s\}} V(t -\xi) \la F_{\la} R(\la, A) \La\left(\varphi_{r(\xi)-s}\right) d \xi\\
    &=&\int_{s}^tV(t-\xi)\lambda F_\lambda R(\lambda,A)\Lambda(U_\Lambda(r(\xi),s)\vp)d\xi.
\end{eqnarray*} 

which shows  \eqref{equ4.4}, and establishes (ii).

\medskip

\noindent{\bf(iii)} We define the function $y(t) = [U_{\La}(t, s) \vp](0)$ for $t \geq s$, and $y(\xi) = \vp(\xi -s)$ for $\xi \leq s$. It follows from
\eqref{equ4.12} that $y_{t} = U_{\La}(t, s) \vp$. Hence,  using \eqref{equ4.4}, we obtain
\begin{eqnarray*}
y(t) & = & [U_{\La}(t, s) \vp](0) \\
& = & [U_{0}(t, s) \vp](0) + \left[\lim_{\la \to \infty} \int_{s}^{t} V(t -\xi) \la F_{\la} R(\la, A) \La \left(U_{\La}(r(\xi), s)  \vp \right)
d \xi \right](0) \\
& = & T(t -s) \vp(0) + \lim_{\la \to \infty} \int_{s}^{t} T(t -\xi) \la  R(\la, A) \La \left(U_{\Lambda}(r(\xi),s)\vp\right) d \xi \\
& = & T(t -s) \vp(0) + \int_{s}^{t} T(t -\xi)  \La \left(U_{\Lambda}(r(\xi),s)\vp\right) d \xi.
\end{eqnarray*}
Proceeding as in (ii), we split the integral into two regions:
\begin{equation*}
\begin{split}
    \int_{s}^{t} T(t -\xi)  \La \left(U_{\Lambda}(r(\xi),s)\vp\right) d \xi= & \int_{\{\xi \in [s,t]\,:\, r(\xi)\geq s\}} T(t -\xi)  \La \left(U_{\Lambda}(r(\xi),s)\vp\right) d \xi \\
    & +\int_{\{\xi \in [s,t]\,:\, r(\xi)< s\}}T(t -\xi)  \La \left(U_{\Lambda}(r(\xi),s)\vp\right) d \xi.
\end{split}
\end{equation*}
Clearly, the first integral satisfies
\begin{eqnarray*}
    \int_{\{\xi \in [s,t]\,:\, r(\xi)\geq s\}}T(t -\xi)  \La \left(U_{\Lambda}(r(\xi),s)\vp\right) d \xi=\int_{\{\xi \in [s,t]\,:\, r(\xi)\geq s\}}T(t -\xi)  \La \left(y_{r(\xi)}\right) d \xi.
\end{eqnarray*}
On the other hand, for every $\xi \in [s,t]$ such that $r(\xi)<s$, we have
\[
 \vp_{r(\xi)-s}(\theta)=\vp(r(\xi)+\theta-s)=y(r(\xi)+\theta)=y_{r(\xi)}(\theta),\quad \text {for all $\theta \leq 0$.} 
\]
Hence, the second integral can be written as follows:
\begin{eqnarray*}
    \int_{\{\xi \in [s,t]\,:\, r(\xi)< s\}}T(t -\xi)  \La \left(U_{\Lambda}(r(\xi),s)\vp\right) d \xi&=\displaystyle\int_{\{\xi \in [s,t]\,:\, r(\xi)< s\}}T(t -\xi)  \La \left(\vp_{r(\xi)-s}\right) d \xi\\
    &=\displaystyle\int_{\{\xi \in [s,t]\,:\, r(\xi)< s\}}T(t -\xi)  \La \left(y_{r(\xi)}\right) d \xi,
\end{eqnarray*}
which shows that 
\begin{eqnarray*}
    y(t)=T(t-s)\vp(0)+\int_s^tT(t-\xi)\Lambda\left(y_{r(\xi)}\right), \quad t\geq s.
\end{eqnarray*}
Therefore, $y$ is the mild solution of problem \eqref{equ4.1}-\eqref{equ4.2}, which completes the proof.

\end{proof}

\begin{remark}\rm
As an immediate consequence of Theorem~\ref{P4.1} we have that $U(t,s) = U_{\La}(t,s)$, for all $(t,s) \in \Delta$.
Moreover, for $t -s \geq \tau$, it  follows from \eqref{equ4.4} that
\begin{equation} \label{equ4.20}
[U_{\La}(t, s) \vp](\ta) = \left\{ \begin{array}{lcl} T(t + \ta - s) \vp(0) & & \\
+ \displaystyle \int_{s}^{t + \ta} T(t + \ta -\xi)  \La \left(U_{\La}(r(\xi), s)  \vp \right) d \xi, & & t + \ta \geq s, \\\\
\vp(t + \ta -s), & & t + \ta \leq s. \end{array} \right.
\end{equation}
Also, for a function $h \in L^{1}([0, a], X)$, from  Lebesgue's dominated convergence theorem we infer that
\begin{align}\label{eq 5.11}
    \lim_{\la \to \infty} \int_{r(t)}^{t} U_{\La}(r(t), \xi) \la F_{\la} R(\la, A) h(\xi) d \xi = 0,
\end{align}

where we have extended $h(\xi) = 0$ for $\xi \leq 0$,
property that will be used repeatedly in the development that follows.
\end{remark}

\bigskip

\subsection{The variation of constants formula}

We return to study the nonhomogeneous problem \eqref{equ4.1}-\eqref{equ4.2}.
Initially we establish a slight generalization of Gronwall's lemma.
\begin{lema} \label{L4.6} Let $\al: [\tau + s, a] \to [0, \infty)$ be a continuous function such that
$\al(\xi) \leq C_{1}$ for  $\tau + s \leq \xi \leq s$, and
\[
\al(t) \leq C_{1} + C_{2} \int_{s}^{t} \al(r(\xi)) d \xi, \, t \geq s.
\]
Then $\al(t) \leq C_{1} e^{C_{2}(t -s)}$ for $t \geq s$.
\end{lema}
\begin{proof} We define $\beta(\sa) = \sup_{\tau + s \leq \xi \leq \sa} \al(\xi)$. Since $\beta$ is a non decreasing function,
and $r(\xi) \leq \xi$, we can affirm that
\begin{eqnarray*}
\al(t) & \leq & C_{1} + C_{2} \int_{s}^{t} \beta(r(\xi)) d \xi \\
 & \leq & C_{1} + C_{2} \int_{s}^{t} \beta(\xi) d \xi
\end{eqnarray*}
for $t \geq s$. Hence
\[
\beta(t) \leq  C_{1} + C_{2} \int_{s}^{t} \beta(\xi) d \xi,
\]
which implies that
\[
\al(t) \leq \beta(t) \leq C_{1} e^{C_{2}(t -s)}
\]
for $t \geq s$.

\end{proof}

\begin{lema} \label{L4.5} Let $h \in L^{1}([0, a], X)$, then the limit
\[
\lim_{\la \to \infty} \int_{s}^{t} U_{\La}(t, \xi) \la F_{\la} R(\la, A) h(\xi) d \xi
\]
exists in $C_{0}((- \infty, 0], X)$, and the convergence is uniform for $(t,s) \in \Delta$.
\end{lema}
\begin{proof} Assume initially that $h \in C([0, a], X)$. For $\la$ large enough, following \cite{BMM},
we abbreviate the notation by writing
\[
Z_{\la}(t,s) = \int_{s}^{t} U_{\La}(t, \xi) \la F_{\la} R(\la, A) h(\xi) d \xi.
\]
We next extend the domain of $h$ by defining $h(\xi) = 0$ for $\xi \leq 0$.
For $t + \ta \geq s$, using \eqref{equ4.20}, we have
\begin{align*}
& Z_{\la}(t,s)(\ta)   =  \int_{s}^{t + \ta} [U_{\La} (t, \xi) \la F_{\la} R(\la, A) h(\xi)](\ta) d \xi +
\int_{t + \ta}^{t} [U_{\La} (t, \xi) \la F_{\la} R(\la, A) h(\xi)](\ta) d \xi \\
= & \int_{s}^{t + \ta } T(t + \ta - \xi) \la R(\la, A) h(\xi) d \xi \\
&  + \int_{s}^{t + \ta } \int_{\xi}^{t + \ta } T(t + \ta - \sa)
\La( U_{\La} (r(\sa), \xi) \la F_{\la} R(\la, A) h(\xi)) d \sa   d \xi \\
&  + \int_{t + \ta}^{t} \la e^{\la(t + \ta - \xi)} R(\la, A) h(\xi) d \xi.
\end{align*}
On the other hand, changing the order of integration, we obtain
\begin{eqnarray*}
& & \int_{s}^{t + \ta } \int_{\xi}^{t + \ta } T(t + \ta - \sa) \La( U_{\La} (r(\sa), \xi) \la F_{\la} R(\la, A) h(\xi)) d \sa d \xi \\
& = & \int_{s}^{t + \ta } \int_{s}^{\sa} T(t + \ta - \sa) \La( U_{\La} (r(\sa), \xi) \la F_{\la} R(\la, A) h(\xi)) d \xi d \sa.
\end{eqnarray*}
Consequently, because $r(\sigma) \leq \sigma$ and $\Lambda$ is a bounded linear operator, it follows that
\begin{eqnarray*}
& & \int_{s}^{t + \ta } \int_{\xi}^{t + \ta } T(t + \ta - \sa) \La( U_{\La} (r(\sa), \xi) \la F_{\la} R(\la, A) h(\xi)) d \sa d \xi \\
& = & \int_{s}^{t + \ta }  \int_{s}^{r(\sa)} T(t + \ta - \sa) \La( U_{\La} (r(\sa), \xi) \la F_{\la} R(\la, A) h(\xi)) d \xi d \sa \\
& & + \int_{s}^{t + \ta } \int_{r(\sa)}^{\sa} T(t + \ta - \sa) \La( U_{\La} (r(\sa), \xi) \la F_{\la} R(\la, A) h(\xi)) d \xi  d \sa \\
& = & \int_{s}^{t + \ta }   T(t + \ta - \sa)\int_{s}^{r(\sa)} \La( U_{\La} (r(\sa), \xi) \la F_{\la} R(\la, A) h(\xi)) d \xi d \sa \\
& & + \int_{s}^{t + \ta } \int_{r(\sa)}^{\sa} T(t + \ta - \sa) \La( U_{\La} (r(\sa), \xi) \la F_{\la} R(\la, A) h(\xi)) d \xi  d \sa \\
& = & \int_{s}^{t + \ta }   T(t + \ta - \sa)\La\left(\int_{s}^{r(\sa)}  U_{\La} (r(\sa), \xi) \la F_{\la} R(\la, A) h(\xi) d \xi  \right)d\sigma \\
& & + \int_{s}^{t + \ta } \int_{r(\sa)}^{\sa} T(t + \ta - \sa) \La( U_{\La} (r(\sa), \xi) \la F_{\la} R(\la, A) h(\xi)) d \xi  d \sa \\
& = & \int_{s}^{t + \ta }  T(t + \ta - \sa) \La(Z_{\la}(r(\sa), s))  d \sa \\
& & + \int_{s}^{t + \ta }
\int_{r(\sa)}^{\sa} T(t + \ta - \sa) \La( U_{\La} (r(\sa), \xi) \la F_{\la} R(\la, A) h(\xi)) d \xi  d \sa.
\end{eqnarray*}
Therefore, for $t + \ta \geq s$, we have
\begin{eqnarray*}
\lefteqn{Z_{\la}(t,s)(\ta)  = \int_{s}^{t + \ta } T(t + \ta - \xi) \la R(\la, A) h(\xi) d \xi
+ \int_{t + \ta}^{t} \la e^{\la(t + \ta - \xi)} R(\la, A) h(\xi) d \xi} \\
& & + \int_{s}^{t + \ta}
\int_{r(\sa)}^{\sa} T(t + \ta - \sa) \La( U_{\La} (r(\sa), \xi) \la F_{\la} R(\la, A) h(\xi)) d \xi  d \sa \\
& & + \int_{s}^{t + \ta }  T(t + \ta - \sa) \La(Z_{\la}(r(\sa), s))  d \sa.
\end{eqnarray*}
On the other hand, for  for $t + \ta \leq s$, we have from \eqref{equ4.20}
\[
Z_{\la}(t,s)(\ta) =  \int_{s}^{t} \la e^{\la(t + \ta - \xi)} R(\la, A) h(\xi) d \xi.
\]
Let $\la, \mu$ large enough. For $t + \ta \geq s$, we obtain
\begin{eqnarray}
\lefteqn{\|Z_{\la}(t,s)(\ta) - Z_{\mu}(t,s)(\ta) \| \leq  \int_{s}^{t + \ta } \|T(t + \ta - \xi) (\la R(\la, A) - \mu R(\mu, A)) h(\xi)\| d \xi} \nonumber \\
& & + \int_{t + \ta}^{t} \|[\la e^{\la(t + \ta - \xi)} R(\la, A) - \mu e^{\mu (t + \ta - \xi)} R(\mu, A) ] h(\xi) \| d \xi \nonumber \\
& & + \int_{s}^{t + \ta} \int_{r(\sa)}^{\sa} \|T(t + \ta - \sa) \La( U_{\La} (r(\sa), \xi) [\la F_{\la} R(\la, A) - \mu F_{\mu} R(\mu, A)] h(\xi)) \|
d \xi  d \sa \nonumber \\
& & + \int_{s}^{t + \ta }  \|T(t + \ta - \sa) \La(Z_{\la}(r(\sa), s) - Z_{\mu}(r(\sa), s)) \| d \sa. \label{equ4.7}
\end{eqnarray}
We now estimate each term on the right hand side of \eqref{equ4.7}.

For the first and third term we use that
\[
\la R(\la, A) h(\xi) - \mu R(\mu, A) h(\xi) \longrightarrow 0, \; \la, \mu \to \infty.
\]
The Lebesgue dominated convergence theorem implies that
\[
\int_{s}^{t + \ta } \| (\la R(\la, A) - \mu R(\mu, A)) h(\xi)\| d \xi \longrightarrow 0, \; \la, \mu \to \infty,
\]
and the convergence   is uniform for $h$ in bounded subsets of $L^{1}([0,a], X)$. To estimate the second term, we use that
there exists a constant $C_{1} \geq 0$ independent of $\xi \leq a$ such that
\[
\|\la R(\la, A) h(\xi)\| \leq C_{1},
\]
for $\la$ large enough. Hence,
\begin{eqnarray*}
& & \int_{t + \ta}^{t} \|[\la e^{\la(t + \ta - \xi)} R(\la, A) - \mu e^{\mu (t + \ta - \xi)} R(\mu, A) ] h(\xi) \| d \xi \\
& &  \leq  C_{1} \int_{t + \ta}^{t} (e^{\la(t + \ta - \xi)} \|h(\xi)\| + e^{\mu(t + \ta - \xi)} \|h(\xi)\|) d \xi,
\end{eqnarray*}
and using again the Lebesgue dominated convergence theorem we infer  that
\[
\int_{t + \ta}^{t} \|[\la e^{\la(t + \ta - \xi)} R(\la, A) - \mu e^{\mu (t + \ta - \xi)} R(\mu, A) ] h(\xi) \| d \xi \longrightarrow 0, \; \la, \mu \to \infty,
\]
and the convergence is uniform for $h$ in bounded subsets of $L^{1}([0,a], X)$.
For the fourth term, we have
\[
\int_{s}^{t + \ta }  \|T(t + \ta - \sa) \La(Z_{\la}(r(\sa), s) - Z_{\mu}(r(\sa), s)) \| d \sa \leq C_{2} \int_{s}^{t + \ta } \|Z_{\la}(r(\sa), s) - Z_{\mu}(r(\sa), s) \| d \sa
\]
for some constant $C_{2}$. Combining these estimates, and applying Lemma~\ref{L4.6}, we conclude that
\[
\|Z_{\la}(t,s)(\ta) - Z_{\mu}(t,s)(\ta) \| \longrightarrow 0, \; \la, \mu \to \infty,
\]
uniformly for $\ta \leq 0$. Moreover, the convergence is uniform on $h$ for $h$ in bounded sets of  $L^{1}([0,a], X)$.

Finally, returning to the general case, for $h \in L^{1}([0,a], X)$, the assertion is a consequence of the fact that $C([0,a], X)$ is dense $L^{1}([0,a], X)$,
and also from the fact that the convergence for continuous functions  $h$ is uniform  for $h$ in bounded sets of  $L^{1}([0,a], X)$.

\end{proof}

We are in a position to establish the main result of this section, which establishes the mild solution of problem
\eqref{equ4.1}-\eqref{equ4.2} in the space $\bb_{\tau}$.
\begin{theorem} Let $\vp \in \bb_{\tau}$ and $h \in L^{1}([0, a], X)$. Let $u : [s, a] \to \bb_{\tau}$ be the function given by
\begin{equation}
u(t) = U_{\La}(t, s) \vp +  \lim_{\la \to \infty} \int_{s}^{t} U_{\La}(t, \xi) \la F_{\la} R(\la, A) h(\xi) d \xi. \label{equ4.6}
\end{equation}
Then the function $x:(-\infty,a]\to X$ defined by
\[
x(t) = \left\{ \begin{array}{lcl} u(t)(0), & & t \geq s, \\
\vp(t -s), & & t \leq s, \end{array} \right.
\]
is the mild solution of  problem \eqref{equ4.1}-\eqref{equ4.2}.
\end{theorem}
\begin{proof} Assume that $t\geq  s$. It is follows from \eqref{equ4.20} that
\begin{align*}
\lefteqn{x(t)  =
T(t  - s) \vp(0) +  \int_{s}^{t } T(t -\xi)  \La \left(U_{\La}(r(\xi), s)  \vp \right) d \xi} \\
&  + \lim_{\la \to \infty} \int_{s}^{t } T(t -\xi)  \la R(\la, A) h(\xi) d \xi \\
&   + \lim_{\la \to \infty} \int_{s}^{t } \int_{\xi}^{t } T(t  - \sa)
\La\left(U_{\La}(r(\sa), \xi) \la F_{\la} R(\la, A) h(\xi) \right) d \sa d \xi 
\end{align*}

Using the fact that.
\begin{eqnarray*}
    \lim_{\la \to \infty} \int_{s}^{t } T(t -\xi)  \la R(\la, A) h(\xi) d \xi=\int_s^{t}T(t-\xi)h(\xi)d\xi
\end{eqnarray*}
and
\begin{eqnarray*}
    & & \int_{s}^{t } \int_{\xi}^{t } T(t  - \sa) \La( U_{\La} (r(\sa), \xi) \la F_{\la} R(\la, A) h(\xi)) d \sa d \xi \\
& = & \int_{s}^{t  } \int_{s}^{\sa} T(t  - \sa) \La( U_{\La} (r(\sa), \xi) \la F_{\la} R(\la, A) h(\xi)) d \xi d \sa,
\end{eqnarray*}
it follows that
\begin{align*}
&    [u(t)](0)= 
 T(t - s) \vp(0) +  \int_{s}^{t } T(t + \ta -\xi) h(\xi) d \xi +
  \int_{s}^{t } T(t  -\xi)  \La \left(U_{\La}(r(\xi), s)  \vp \right) d \xi \\
&  +  \lim_{\la \to \infty} \int_{s}^{t} T(t - \sa) \La\left( \int_{s}^{\sa}
U_{\La}(r(\sa), \xi) \la F_{\la} R(\la, A) h(\xi)  d \xi \right) d \sa \\
 = &
T(t  - s) \vp(0) +  \int_{s}^{t} T(t  -\xi) h(\xi) d \xi \\
&  + \int_{s}^{t + \ta} T(t + \ta -\sa)  \La \left( U_{\La}(r(\sa), s)  \vp
 +  \lim_{\la \to \infty} \int_{s}^{\sa} U_{\La}(r(\sa), \xi) \la F_{\la} R(\la, A) h(\xi) d \xi \right)  d \sa \\
 = &
T(t  - s) \vp(0) +  \int_{s}^{t } T(t -\xi) h(\xi) d \xi \\
&  +\int_{s}^{t } T(t  -\sa)  \La \left( U_{\La}(r(\sa), s)  \vp
 +  \lim_{\la \to \infty} \int_{s}^{r(\sa)} U_{\La}(r(\sa), \xi) \la F_{\la} R(\la, A) h(\xi) d \xi \right)  d \sa,
\end{align*}
where, in the last equality, we have used equation \eqref{eq 5.11}. To streamline the notation in the subsequent integrals, we partition the interval 
$[s,t]$ into the disjoint sets
\[
I_{s,t}^{+} := \{\sigma \in [s,t] : r(\sigma) \ge s\} 
\quad \text{and} \quad 
I_{s,t}^{-} := \{\sigma \in [s,t] : r(\sigma) < s\}.
\]
With this notation, we split the last integral into two regions:
\begin{eqnarray*}
    \int_{s}^{t } T(t -\sa)  \La \left( U_{\La}(r(\sa), s)  \vp
 +  \lim_{\la \to \infty} \int_{s}^{r(\sa)} U_{\La}(r(\sa), \xi) \la F_{\la} R(\la, A) h(\xi) d \xi \right)  d \sa:=
 I_1+I_2,
\end{eqnarray*}
where
\begin{eqnarray*}
    I_1=\int_{I_{s,t}^{+}} T(t -\sa)  \La \left( U_{\La}(r(\sa), s)  \vp
 +  \lim_{\la \to \infty} \int_{s}^{r(\sa)} U_{\La}(r(\sa), \xi) \la F_{\la} R(\la, A) h(\xi) d \xi \right)  d \sa,
\end{eqnarray*}
and
\begin{eqnarray*}
    I_2=\int_{I_{s,t}^{-}} T(t -\sa)  \La \left( U_{\La}(r(\sa), s)  \vp
 +  \lim_{\la \to \infty} \int_{s}^{r(\sa)} U_{\La}(r(\sa), \xi) \la F_{\la} R(\la, A) h(\xi) d \xi \right)  d \sa.
\end{eqnarray*}
In the first integral, we obtain from the equation \eqref{equ4.6} that
\begin{align*}
    I_1=\int_{I_{s,t}^{+}} T(t -\sa)  \La \left(x_{r(\sigma)}\right)  d \sa.
\end{align*}
For the second integral, we recall that $U_{\Lambda}(r(\sigma),s)\varphi=\varphi_{r(\sigma)-s}$, for $r(\sigma)<s$. Then,
\begin{align*}
    I_2&=\int_{I_{s,t}^{-}} T(t -\sa)  \La \left(\varphi_{r(\sigma)-s}
 -  \lim_{\la \to \infty} \int_{r(\sa)}^{s} \la (F_{\lambda}R(\lambda,A)h(\xi))_{r(\sigma)-\xi}d\xi \right)  d \sa\\
 &=\int_{I_{s,t}^{-}} T(t -\sa)  \La (\varphi_{r(\sigma)-s})  d \sa\\
 &=\int_{I_{s,t}^{-}} T(t -\sa)  \La (x_{r(\sigma)})  d \sa.
\end{align*}
Therefore, for all $t \geq s$,
\begin{eqnarray*}
    x(t)=T(t-s)\varphi(0)+\int_s^tT(t-\xi)\Lambda(x_{r(\xi)})d\xi+\int_s^tT(t-\xi)h(\xi)d\xi,
\end{eqnarray*}
and $x$
is the mild solution of problem \eqref{equ4.1}-\eqref{equ4.2}.

\end{proof}

\begin{remark}
	We emphasize that expression \eqref{equ4.6} is the variation of constants formula for problem \eqref{equ4.1}-\eqref{equ4.2}.
\end{remark}

\section{Acknowledgements}
C. Carrasco was partially supported by ANID Doctorado Nacional No. 21240764; C. A. Gallegos was partially supported by ANID-FONDECYT Iniciaci\'on No. 11260638; M. F. Pinaud was partially supported by DICYT-USACH grant 042632PC-POSTDOC.

\section{Conflict of interest statement}
	On behalf of all authors, the corresponding author states that there is no conflict of interest.
 
\section{Data Availability Statement}
Data sharing not applicable to this article as no datasets were generated or analyzed during the current study.

\end{document}